\documentclass[a4paper,12pt]{amsart}

\usepackage[english]{babel}
\usepackage[english]{translator}
\usepackage[T1]{fontenc}
\usepackage[utf8]{inputenc}

\usepackage{amsmath}
\usepackage{amssymb}
\usepackage{amsfonts}
\usepackage{amstext}
\usepackage{amsthm}
\usepackage{bbm}

\usepackage{graphicx}
\usepackage{color}
\usepackage{psfrag}
\usepackage{subfigure}
\usepackage{float}

\usepackage{marvosym}
\usepackage{enumitem}

\usepackage{hyperref}

\usepackage[
  hmarginratio={1:1},     
  vmarginratio={1:1},     
  textwidth=460pt,        
  heightrounded,          
  bottom=3.65cm,
  top=3.65cm,
]{geometry}

\hypersetup{
  pdffitwindow=false,
  pdfhighlight=/O,
  pdfnewwindow,
  colorlinks=true,
  citecolor=red,            
  linkcolor=blue,             
  menucolor=blue,            
  urlcolor=blue,             
  pdfpagemode=UseOutlines,
  bookmarksnumbered=true,
  linktocpage,
  pdftitle={Fredholm Properties and $L^p$-Spectra of Localized Rotating Waves in Parabolic Systems},
  pdfsubject={Fredholm Properties and $L^p$-Spectra of Localized Rotating Waves in Parabolic Systems},
  pdfauthor={Wolf-J"urgen Beyn, Denny Otten},
  pdfkeywords={Rotating wave, systems of reaction-diffusion equations, Fredholm theory, $L^p$-spectrum, essential spectrum, point spectrum, Ornstein-Uhlenbeck operator.},
  pdfcreator={pdflatex},
  pdfproducer={LaTeX with hyperref}
}

\newcommand{\cL}{\mathcal{L}}             
\newcommand{\cO}{\mathcal{O}}             
\newcommand{\CC}{\mathbb{C}}              
\newcommand{\RR}{\mathbb{R}}              
\newcommand{\NN}{\mathbb{N}}               
\newcommand{\B}{\mathcal{B}}             
\newcommand{\fg}{\mathfrak{g}}             
\newcommand{\one}{\mathbbm{1}}             

\DeclareMathOperator{\Real}{Re}

\makeatletter
\renewcommand{\@secnumfont}{\bfseries}
\makeatother
\makeatletter
  \def\section{\@startsection{section}{1}%
    \z@{.7\linespacing\@plus\linespacing}{.5\linespacing}%
    {\normalfont\LARGE\bfseries}}
\makeatother
\makeatletter
\def\@seccntformat#1{%
  \protect\textup{%
    \protect\@secnumfont
    \expandafter\protect\csname format#1\endcsname 
    \csname the#1\endcsname
    \protect\@secnumpunct
  }%
}

\newcommand{\sect}
{
  \setcounter{equation}{0}
  \setcounter{figure}{0}
  \section
}

\theoremstyle{definition}
\newtheorem{definition}{Definition}[section]
\newtheorem{assumption}[definition]{Assumption}
\newtheorem{example}[definition]{Example}
\newtheorem{remark}[definition]{Remark}

\theoremstyle{plain}
\newtheorem{theorem}[definition]{Theorem}

\allowdisplaybreaks

\begin{document}
\title[Computation and stability of waves\\ in equivariant evolution equations]{Computation and stability of waves\\ in equivariant evolution equations}
\setlength{\parindent}{0pt}
\vspace*{0.75cm} 
\begin{center}
\normalfont\huge\bfseries{\shorttitle}\\
\vspace*{0.25cm} 
\end{center}

\vspace*{0.8cm}
\noindent
\begin{minipage}[t]{0.99\textwidth}
  \begin{center}
        {\bf Wolf-J{\"u}rgen Beyn}\footnotemark[1]${}^{,}$\footnotemark[3] 
        \hspace*{1.0cm}
        \textbf{Denny Otten}\footnotemark[2]${}^{,}$\footnotemark[3]\\
        Department of Mathematics, Bielefeld University \\
33501 Bielefeld, Germany
\end{center}
\end{minipage}\\

\footnotetext[1]{e-mail: \textcolor{blue}{beyn@math.uni-bielefeld.de}, \\
  homepage: \url{http://www.math.uni-bielefeld.de/~beyn/AG\_Numerik/}.}
\footnotetext[2]{e-mail: \textcolor{blue}{dotten@math.uni-bielefeld.de}, \\
 homepage: \url{http://www.math.uni-bielefeld.de/~dotten/}.}
\footnotetext[3]{supported by CRC 701 'Spectral Structures and Topological Methods in Mathematics',  Bielefeld University}

\vspace*{0.8cm} 
\noindent
\hspace*{5.4cm}
Date: \today
\normalparindent=12pt

\vspace{1.05cm} 
\noindent
\begin{center}
\begin{minipage}{0.9\textwidth}
  {\small
  \textbf{Abstract.} 

  Travelling and rotating waves are ubiquitous phenomena observed in time dependent
PDEs modelling the combined effect of dissipation and non-linear interaction.
From an abstract viewpoint they appear as relative equilibria of an equivariant 
evolution equation. In numerical computations the freezing method takes advantage 
of this structure by splitting the evolution of the PDE into the dynamics on the 
underlying Lie group and on some reduced phase space. The approach raises a series of questions which were answered to a certain degree by the project: 
linear stability implies non-linear (asymptotic) stability, persistence of
stability under discretisation, analysis and 
computation of spectral structures, first versus second order evolution systems, 
well-posedness of partial differential algebraic equations, spatial decay of wave 
profiles and truncation to bounded domains, analytical and numerical 
treatment of wave interactions, relation to connecting orbits in dynamical systems. A further numerical problem related to this topic will
be discussed, namely  the solution of non-linear eigenvalue problems via
a contour method.}
\end{minipage}
\end{center}
\vspace{0.3cm}



\noindent
\textbf{Key words.} Equivariant evolution equations, relative equilibria,
freezing method, rotating and traveling waves, asymptotic stability, reaction-diffusion systems,
Hamiltonian PDEs, nonlinear eigenvalue problems.
\vspace{0.3cm}

\noindent
\textbf{AMS subject classification.} 37Lxx, 65J08, 74J30, 65L15, 35B40.


%
%
\newpage
\sect[Equivariant Evolution Equations]{Equivariant Evolution Equations}
\label{b3-s1}
\subsection{Abstract Setting}\label{b3-s1.1}
The overall topic of the project is the numerical analysis 
of evolution equations which may be written in the abstract form
\begin{equation} \label{b3-e1}
u_t = F(u), \quad t\ge 0,
\end{equation}
where the solution $u:[0,T) \rightarrow X$, $t \mapsto u(t)$ is defined on
a real interval $[0,T),0 < T \le \infty$ has values in a Banach
space $X$, and time derivative $u_t$. The map $F: Z\subseteq X \rightarrow X$ is a vector field
defined on a dense subspace $Z$ of $X$.
\index{evolution equation} \index{equivariance}
The additional structure is described in terms of a Lie group $G$ 
\index{Lie group} of dimension
$n=\dim(G)< \infty$ which acts on $X$ via a homomorphism into the general linear
group $GL(X)$ of homeomorphisms on $X$:
\begin{equation} \label{b3-e2}
a:G \rightarrow GL(X), \quad \gamma \mapsto a(\gamma).
\end{equation}
For the images we use the synonymous notation $a(g)u=a(g,u),g\in G, u \in X$.

We assume that equation \eqref{b3-e1} is equivariant with respect to this group
action, i.e. the vector field $F$ has the following property
\begin{equation} \label{b3-e3}
F(a(\gamma)u) = a(\gamma)F(u) \quad \forall \gamma \in G \; \forall u \in Z,
\end{equation}
where we have assumed $a(\gamma)Z \subseteq Z$ for all $\gamma \in G$.

In Sections \ref{b3-s2.2} and \ref{b3-s4} we will deal with several classes
of partial differential equations which fit into this general setting.
All of them are formulated for functions on a Euclidean space $\RR^d$ where the action is
caused by the special Euclidean group $SE(d)$ acting via rotations and translations
on their arguments or on their values.
\begin{remark} \label{b3-rem1}
For some applications even this framework is not sufficient. For example,
travelling fronts which have finite but non-zero limits at infinity, do not
lie in any of the usual Lesbesgue or Sobolev spaces, but in an affine
space. To cover such cases, but also more general PDEs on manifolds, one can
generalise the whole approach to Banach manifolds $X$,
where $F$ is a vector field defined on a submanifold $Z$ of $X$ mapping into the tangent bundle $TX$, 
and $a$ takes values in the space of diffeomorphisms $\mathrm{Diff}(X,X)$.
Equivariance \eqref{b3-e3} is then expressed as $F(a(\gamma,u))=d_u[a(\gamma,u)]F(u)$,
where $d_u[a(\gamma,u)]:T_uX \rightarrow T_{a(\gamma,u)}X$ denotes the tangent map.
For the sake of simplicity we will not pursue this generalisation
here (see \cite{b3-R10}).
\end{remark}

For this article it is sufficient to work with a simple notion of a strong solution of a Cauchy
problem instead of dealing with weak solutions in time and mild solutions
in space.
\begin{definition} \label{b3-d1}
A function $u\in C^1([0,T),X) \cap C([0,T),Z)$ satisfying
\begin{equation} \label{b3-e4} u_t = F(u), t \in [0,T), \quad u(0)=u_0 \in Z,
\end{equation}
is called a strong solution of the Cauchy problem \eqref{b3-e4}.
\end{definition}
\index{strong solution}
In the following we will always assume that a strong solution of \eqref{b3-e4}
exists locally, i.e. on some interval $[0,T), T>0$, and that it is unique.
For applications to PDEs it is typical that the group action is only
differentiable for smooth functions. Therefore, we impose the following condition.  
\begin{assumption} \label{b3:a1}
For any $u \in X$, resp. $u\in Z$ the mapping 
\begin{equation*} \label{b3-e5}
a(\cdot)u: G \rightarrow X , \quad \gamma \mapsto a(\gamma)u
\end{equation*}
is continuous resp. continuously differentiable with derivative
\begin{equation} \label{b3-e6}
d_{\gamma}[a(\gamma)u]: T_{\gamma}G \rightarrow X, \quad 
\mu \mapsto d_{\gamma}[a(\gamma)u] \mu.
\end{equation}
\end{assumption}
In case $\gamma = \one$ the tangent space $T_{\one}G$ may be identified with
the Lie algebra $\fg$ of $G$, and we have $d_{\gamma}[a(\one)u]:\fg \rightarrow
X$.
\index{Lie algebra}

\subsection{Relative Equilibria}
\label{b3-s1.2}
Relative equilibria are special solutions of \eqref{b3-e1} which lie
in the group orbit of a single element.
\begin{definition} \label{b3-d2}
A pair $v_{\star}\in Z, \gamma_{\star}\in C^1([0,\infty),G)$ is called a relative
equilibrium of \eqref{b3-e1} if $\gamma_{\star}(0)=\one $ and $u_{\star}(t) =a(\gamma_{\star}(t))v_{\star},
t \ge 0$ is a strong solution of \eqref{b3-e4} with $u_0=v_{\star}$.
\end{definition}
\index{relative equilibrium}
In some references (see e.g. \cite{b3-CL00}) the whole group orbit
$\cO_G(v_{\star})= \{ a(g)v_{\star}:g \in G \}$ is called a relative equilibrium.
However, we include the path $t \rightarrow \gamma_{\star}(t)$ on the group
as part of our definition since it will be relevant for both the stability analysis
and numerical computations. The following theorem shows that the path may always be written
as $\gamma_{\star}(t) =\exp(t \mu_{\star})$ for some $\mu_{\star} \in \fg$.
Recall that $\exp:\fg \rightarrow G$ is the exponential function and that
$\gamma_{\star}(t)=\exp(t \mu_{\star}), t \in \RR$ is the unique solution of
the Cauchy problem
\begin{equation} \label{b3-e6a}
\gamma_{\star}'(t) = dL_{\gamma_{\star}(t)}(\one) \mu_{\star}, \quad \gamma_{\star}(0)=\one,
\end{equation}
where $L_{\gamma}g= \gamma \circ g,g\in G$ denotes the multiplication by
$\gamma$ from the left.
The vector field on the right-hand side of \eqref{b3-e6a} is often simply
written as $\gamma_{\star}(t) \mu_{\star}$, but in analogy to \eqref{b3-e6}
we keep the slightly clumsier
notation $dL_{\gamma_{\star}(t)}(\one) \mu_{\star}$ for clarity. 
\begin{theorem} \label{b3:t1} Let Assumption \ref{b3:a1} hold. Then
for every relative equilibrium $v_{\star}\in Z, \gamma_{\star}\in C^1([0,\infty),G)$ there exists $\mu_{\star}\in \fg$ such that
\begin{align} \label{b3-e7}
0= & F(v_{\star})- d_{\gamma}[a(\one)v_{\star}] \mu_{\star} \\
\label{b3-e8}
a(\gamma_{\star}(t))v_{\star}= & a(\exp(t \mu_{\star}))v_{\star}.
\end{align}
Conversely, let $v_{\star}\in Z, \mu_{\star} \in \fg$ solve \eqref{b3-e7},
then $v_{\star}$ and $\gamma_{\star}(t) = \exp(t \mu_{\star}),t\ge 0$ are a relative
equilibrium.
\end{theorem}
Given $v_{\star}$ and $\gamma_{\star}(\cdot)$, then uniqueness of $\mu_{\star}$ will follow from \eqref{b3-e8} in the first part of the theorem if 
the stabiliser $H(v_{\star})$ of $v_{\star}$ is simple, i.e.
 \begin{equation}\label{b3-e8a}
H(v_{\star})= \{\gamma\in G:a(\gamma)v_{\star}=v_{\star}\} =\{\one\}.
\end{equation} \index{stabiliser} 
\noindent
But even then, equation \eqref{b3-e7} does not determine the pair 
$(v_{\star},\mu_{\star})$
uniquely, since relative equilibria always come in families. More precisely,
Definition \ref{b3-d2} and the equivariance \eqref{b3-e3} show that any relative equilibrium
$v_{\star},\gamma_{\star}$ of \eqref{b3-e4} is accompanied by a family 
 $(w(g),\gamma(g,\cdot)), g \in G$ of relative equilibria given by
\begin{equation} \label{b3-e8b}
 w(g)=a(g)v_{\star}, \quad \gamma(g,t)= g \circ \gamma_{\star}(t) \circ g^{-1}, \quad g \in G,
\end{equation}
see \cite{b3-CS07} for related results. This will be important for the
stability analysis in Section \ref{b3-s3}.

\subsection{Wave Solutions of PDEs}
\label{b3-s1.3}
Two important classes of semi-linear evolution
equations which fit into the above setting and to which our results apply, are
the following
\begin{equation} \label{b3-e9}
u_t = Au_{xx} + f(u,u_x), \quad u(x,t) \in \RR^m, \quad  x \in \RR, t \ge 0,
\quad u(\cdot,0)=u_0,
\end{equation}
\begin{equation} \label{b3-e10}
u_t = A \Delta u + f(u), \quad u(x,t) \in \RR^m, \quad x \in \RR^d, t\ge 0,
\quad u(\cdot,0)=u_0.
\end{equation}
In both cases $A\in \RR^{m,m}$ is assumed to have spectrum $\sigma(A)$ with
$\Real(\sigma(A))\ge 0$. Note that $\Real(\sigma(A))>0$ leads to parabolic
systems while $\sigma(A)\subseteq i \RR$  occurs for
Hamiltonian PDEs. Intermediate cases with $\sigma(A) \subset (\{0\} \cup
\{\Real z >0\})$ generally belong to hyperbolic or parabolic-hyperbolic
mixed systems.
\index{parabolic equation} \index{Hamiltonian PDE} \index{parabolic-hyperbolic equation}
The non-linearities $f:\RR^{2m} \rightarrow \RR^m$ in \eqref{b3-e9} resp.
$f:\RR^{m} \rightarrow \RR^m$ in \eqref{b3-e10} are assumed to be
sufficiently smooth and to satisfy $f(0,0)=0$ resp. $f(0)=0$.

In case of \eqref{b3-e9} the Lie group is  $(G,\circ)=(\RR,+)$ acting on $X=L^2(\RR,\RR^m)$
by the shift $[a(\gamma)u](x) = u(x-\gamma), x\in \RR, u\in X$. With
$F(u)=Au_{xx}+f(u,u_x)$ for $u \in Z = H^2(\RR,\RR^m)$, equivariance is easily
verified and $F(u) \in X$ follows from the Sobolev embedding 
$H^1(\RR,\RR^m) \subseteq L^{\infty}(\RR,\RR^m)$ and $f(0,0)=0$. For the
derivative we find
\begin{equation*} \label{b3-e11}
d_{\gamma}[a(\one) v] \mu= -v_x \mu, \quad \mu \in \fg=\RR, \quad v\in H^1(\RR,\RR^m).
\end{equation*}
Relative equilibria then turn out to be travelling waves
\begin{equation} \label{b3-e10a}
u_{\star}(x,t) =v_{\star}(x- \mu_{\star}t), x\in \RR,t\ge 0,
\end{equation}
where the pair $(v_{\star},\mu_{\star})$ solves the second order system
from \eqref{b3-e7}
\begin{equation*} \label{b3-e12}
0= Av_{\star,\xi \xi}+\mu_{\star}v_{\star,\xi} + f(v_{\star},v_{\star,\xi}), \quad
v(\xi) \in \RR^m, \xi \in \RR.
\end{equation*}
In fact, our simplified abstract approach only covers pulse solutions
(defined by $v_{\star}(\xi),v_{\star,\xi}(\xi) \rightarrow 0$ as $\xi \rightarrow \pm
\infty$), whereas fronts need the setting of manifolds,
see Remark \ref{b3-rem1}.

In the multi-dimensional case \eqref{b3-e10} the phase space is
$X=L^{2}(\RR^d,\RR^m)$, and we aim at equivariance w.r.t.
the special Euclidean group $G =\mathrm{SE}(d)=\mathrm{SO}(d) \ltimes \RR^d$.
It is convenient to  represent
$\mathrm{SE}(d)$ in $GL(\RR^{d+1})$ as 
\begin{equation} \label{b3-e13}
\mathrm{SE}(d) = \left\{ \begin{pmatrix} Q & b \\ 0 & 1 \end{pmatrix}:
Q \in \RR^{d,d}, Q^{\top}Q=I_d, \det(Q)=1, b \in \RR^d \right\},
\end{equation}
where the group operation is matrix multiplication. We represent the Lie algebra
$\mathfrak{se}(d)=\mathfrak{so}(d)\times \RR^d$  accordingly
\begin{equation} \label{b3-e14}
\mathfrak{se}(d)= \left\{\begin{pmatrix} S & a \\ 0 & 0 \end{pmatrix}:
S \in \RR^{d,d}, S^{\top}=-S, a \in \RR^d \right\}.
\end{equation}
The action on
functions $u \in X$ is defined by
\begin{equation*} \label{b3-e15}
[a(\gamma)u](x) = u(Q^{\top}(x-b)), \quad x \in \RR^d, \gamma=\begin{pmatrix} Q & b \\ 0 & 1 \end{pmatrix} \in \mathrm{SE}(d).
\end{equation*}
The derivative exists for functions
$u \in H^1_{\mathrm{Eucl}}(\RR^d,\RR^m)$ where for $k \ge 1$
\begin{equation*} \label{b3-e16}
H^k_{\mathrm{Eucl}}(\RR^d,\RR^m)=\left\{ u \in H^k(\RR^d,\RR^m):
  \cL_S u\in L^2(\RR^d,\RR^m)\; \forall
S \in \mathfrak{so}(d) \right\},
\end{equation*}
and
\begin{equation*} \label{b3-e15a}
L_Su(x) := u_x(x) S x = \sum_{j,k=1}^d D_ju(x) S_{j,k} x_k,
\quad x \in \RR^d.
\end{equation*}
The derivative of the group action is then given by
\begin{equation*} \label{b3-e16a}
\left(d_{\gamma}[a(\one)v]\mu\right)(x) = -v_x(x)(Sx+ c), \quad x \in \RR^d, \quad
\mu =\begin{pmatrix} S & c \\ 0 & 0 \end{pmatrix} \in \mathfrak{se}(d).
\end{equation*}
Note that the first order operator $\cL_S$ has unbounded coefficients 
and that  the norm in $H^k_{\mathrm{Eucl}}$ is given by
\begin{equation*} \label{b3-e17}
\|u\|^2_{H^k_{\mathrm{Eucl}}}= \|u\|^2_{H^k} + \sup\{
\|\cL_S u \|^2_{L^2}: S \in \mathfrak{so}(d), |S|=1 \}.
\end{equation*}
Setting $Z=H^2_{\mathrm{Eucl}}(\RR^d,\RR^m)$  one
finds $F(u)=A \Delta u + f(u) \in X$ for $u \in Z$ in dimension
$d=2$, since $H^2(\RR^2,\RR^m) \subset L^{\infty}(\RR^2, \RR^m)$ by
Sobolev embedding. But for  $d \ge 3$ one has to impose growth conditions
on $f$ to ensure this. Equivariance follows from the equivariance of the Laplacian
under Euclidean transformations.
Special types of relative equilibria are waves rotating about a centre
$x_{\star} \in \RR^d$:
\begin{equation} \label{b3-e18}
u_{\star}(x,t) = v_{\star}(\exp(-t S_{\star})(x-x_{\star})) ,\quad v_{\star} \in Z,
S_{\star} \in \mathfrak{so}(d).
\end{equation}
 When substituting  $\xi=\exp(-t S_{\star})(x-x_{\star})$ the system \eqref{b3-e7} reads
\begin{equation*} \label{b3-e19}
0= A \Delta v_{\star} + v_{\star,\xi}S_{\star}\xi + f(v_{\star}), \quad \xi \in \RR^d.
\end{equation*}
Several examples of travelling and rotating waves will be dealt with
in Section \ref{b3-s4}.

\sect{The Freezing Method}
\label{b3-s2}
\subsection{The abstract approach}
\label{b3-s2.1}
The idea of the freezing method, set out in \cite{b3-RKML03},\cite{b3-BT04}, is 
to separate the strong solutions of the Cauchy problem \eqref{b3-e4} into a motion
on the group $G$ and on a reduced phase space, just as for the relative
equilibria in Definition \ref{b3-d2}:
\begin{equation} \label{b3-e20}
u(t) = a(\gamma(t))v(t), \quad t \ge 0.
\end{equation}
Let $\gamma\in C^1([0,T),G), \gamma(0)=\one$ and let $u$ be a strong solution of \eqref{b3-e4} and define
$\mu(t):=(dL_{\gamma(t)}(\one))^{-1}\gamma_t(t) \in \fg$ in the Lie algebra
$\fg$ of $G$, then $\gamma,v$ solve the system
\begin{align} 
\label{b3-e21}
v_t(t)=& F(v(t)) - d_{\gamma}[a(\one)v(t)] \mu(t),  & v(0)= u_0,\\
\label{b3-e22}
\gamma_t(t) =& dL_{\gamma(t)}(\one) \mu(t), & \gamma(0)=\one.
\end{align}
Conversely, one can show that a strong solution $u\in C^1([0,T),X) \cap C([0,T),Z)$, $\mu \in  C([0,T),\fg)$, $\gamma \in C^1([0,T),G)$ of 
\eqref{b3-e21},\eqref{b3-e22} leads to a strong solution of \eqref{b3-e4}
via \eqref{b3-e20}.
According to Theorem  \ref{b3:t1} a relative equilibrium $v_{\star},\mu_{\star}$
of \eqref{b3-e1}  is a steady state of the first equation \eqref{b3-e21}.
Following \cite{b3-RKML03}, we call equation \eqref{b3-e22} the reconstruction
equation. 
\index{reconstruction equation}
Due to the extra variables $\gamma\in G$ resp. $\mu\in \fg$, the system \eqref{b3-e21}, \eqref{b3-e22} is not  yet well posed, but needs 
 $n=\dim(G)$ additional algebraic constraints
(called phase conditions) which we write as
\begin{equation} \label{b3-e23}
\psi(v,\mu) = 0. 
\end{equation}
\index{phase condition}
Here $\psi:X\times \fg \rightarrow \fg^{\star}$ (the dual of $\fg$) is a smooth map typically
derived as a necessary condition from a minimisation principle. For example, if $(X,\langle\cdot,\cdot\rangle)$
 is a Hilbert space one can require the distance $\inf_{g \in G}\|v -a(g)\hat{v}\|$
 to the group orbit of a template function $\hat{v} \in X$ (such as $\hat{v}=u_0$) to be minimal at $g=\one$. For $\hat{v} \in Z$ this 
leads to the fixed phase condition
\begin{equation} \label{b3-e24}
\psi_{\mathrm{fix}}(v,\mu) \nu = \langle d_{\gamma}[a(\one)\hat{v}]\nu, v- \hat{v} \rangle=0, \quad \forall \nu \in \fg.
\end{equation}
\index{phase condition!fixed}
An alternative is to minimise $\|v_t\|^2=\|F(v) - d_{\gamma}[a(\one)v] \mu\|^2$ 
with respect to $\mu$ at each time instance, resulting in the orthogonality condition
\begin{equation} \label{b3-e25}
\psi_{\mathrm{orth}}(v,\mu) \nu = \langle d_{\gamma}[a(\one)v]\nu, 
F(v) - d_{\gamma}[a(\one)v] \mu\ \rangle=0
\quad \forall \nu \in \fg.
\end{equation}
This condition requires the group orbit of $v(t)$  to be tangent to 
its time derivative at each time instance. Altogether, equations \eqref{b3-e21} and \eqref{b3-e23}
constitute a partial differential algebraic equation (PDAE) for the functions
$v$ and $\mu$. The reconstruction equation \eqref{b3-e22} decouples from the
PDAE and may be solved in a post-processing step. Condition \eqref{b3-e25} has
a unique solution $\mu$ if $d_{\gamma}[a(\one)v]: \fg \rightarrow X$ is one to one
and then leads to a PDAE of (differentiation) index $1$. Condition 
\eqref{b3-e24} leads to an index $2$ problem, but can be reduced to index $1$ 
by differentiating with respect to $t$ and then inserting \eqref{b3-e21}.

\subsection{Application to Evolution Equations}
\label{b3-s2.2}
In this section we take a closer look at the PDAEs that arise from the freezing
method when applied to the two equations \eqref{b3-e9} and \eqref{b3-e10}.
In Section \ref{b3-s4} we will provide a series of numerical examples
and also discuss the influence of both spatial and temporal discretisation
errors. In the following we restrict to the fixed phase condition
\eqref{b3-e24} which is particularly well-suited near a relative equilibrium
and which admits rather general stability results, see Section \ref{b3-s3}. 
On the other hand the orthogonal phase condition needs no pre-information
and hence can be applied far away from any relative equilibrium. However,
its stability properties are questionable and have only been investigated
in a special case, see \cite{b3-BT07}.

For the one-dimensional system \eqref{b3-e9} with shift equivariance the
freezing ansatz simply reads
\begin{equation} \label{b3-e2.2.1}
u(x,t) = v(x - \gamma(t), t), \quad x\in \RR, t\ge 0, \quad \mu(t) = \gamma_t(t),
\end{equation}
and the corresponding PDAE is given by (cf. \cite{b3-T05})
\begin{equation}\label{b3-equ:s4.1.3}
\begin{aligned}
  v_t &= Av_{\xi\xi} + \mu v_{\xi} + f(v,v_{\xi}), &&v(\cdot,0)=u_0, \\
  0 &= \langle \hat{v}_{\xi},v-\hat{v}\rangle_{L^2(\mathbb{R},\mathbb{R}^m)},\\
  \gamma_t &=\mu, &&\gamma(0)=0,
  \end{aligned}
  \end{equation}
for the unknown quantities $(v,\mu,\gamma)$.
For initial data $u_0$ close to a wave we expect
$v(\cdot,t)\to v_{\star}$, $\mu(t)\to\mu_{\star}$ as $t\to\infty$. Travelling waves
in parabolic  systems  and their stability are analysed in
\cite{b3-H81,b3-S02,b3-VVV94,b3-T05}, and numerical applications of the freezing
method for this case appear in in \cite{b3-BOR14}.

Next, consider the parabolic system \eqref{b3-e10} in several space dimensions.
With the special Euclidean group \eqref{b3-e13} and its Lie algebra \eqref{b3-e14} the freezing system \eqref{b3-e21},\eqref{b3-e22} takes the form 
\begin{equation}\label{b3-equ:s4.1.7}
\begin{aligned}
  v_t &= Av_{\xi\xi} + v_{\xi}(S\xi + c) + f(v), 
  &&v(\cdot,t_0)=u_0, \\
  0 &= \langle \xi_j\hat{v}_{\xi_i} - \xi_i\hat{v}_{\xi_j},v-\hat{v}\rangle_{L^2}, 
  && 0=\langle \hat{v}_{\xi_l},v-\hat{v}\rangle_{L^2}, \\
  \left(\begin{smallmatrix}Q&b\\0&1\end{smallmatrix}\right)_t &= \left(\begin{smallmatrix}Q&b\\0&1\end{smallmatrix}\right)\left(\begin{smallmatrix}S&c
  \\0&0\end{smallmatrix}\right), 
  &&\left(\begin{smallmatrix}Q(t_0)&b(t_0)\\0&1\end{smallmatrix}\right)=I_{d+1},
\end{aligned}
\end{equation}
for the unknown quantities $\left(v,\mu=\left(\begin{smallmatrix}S&c\\0&0\end{smallmatrix}\right),\gamma=\left(\begin{smallmatrix}Q&b\\0&1\end{smallmatrix}\right)\right)$ and indices $1 \le i < j \le d$, $1 \le l \le d$. Since $S(t)$ is skew-symmetric it is sufficient to work with $S_{ij}$, 
$i=1,\ldots,d-1$, $j=i+1,\ldots,d$ and  $c\in \RR^d$ when solving
the reconstruction equation. Numerical methods for differential
equations  on Lie groups may be found in \cite{b3-HLW06}.
If the initial data are close to a stable rotating wave \eqref{b3-e18}
we expect $v(\cdot,t)\to v_{\star}$ and $\left(\begin{smallmatrix}S&c\\0&0\end{smallmatrix}\right)=\mu(t)\to\mu_{\star}=\left(\begin{smallmatrix}S_{\star}& c_{\star} 
\\0&0\end{smallmatrix}\right)\in\mathrm{se}(d)$ as $t\to\infty$. 
Rotating waves in parabolic systems are treated in \cite{b3-FSSW96,b3-FS03}, their non-linear stability (for $d=2$) in \cite{b3-BL08}, and numerical examples in \cite{b3-O14}.
Essential steps for extending non-linear stability to higher space dimensions are done in \cite{b3-BO16a,b3-BO16b, b3-BO18}, which is based on 
previous works \cite{b3-O14,b3-O15,b3-O16a,b3-O16b}.

\subsection{Dynamic Decomposition of Multi-Waves}
\label{b3-s2.3}

Consider a simplified parabolic system \eqref{b3-e9} in one space dimension
\begin{align}
  \label{b3-e2.3.1}
  u_t = Au_{xx} + f(u), &&u(\cdot,0)=u_0,
\end{align}
under the assumptions of Section \ref{b3-s1.3}. Suppose this system
admits several travelling waves $(v_{\star,j},\mu_{\star,j}), j=1,\ldots,N$
with different speeds $\mu_{\star,j}$ and limit behaviour $\lim_{\xi\to\pm\infty}v_{\star,j}(\xi)=v_j^{\pm}$
for $j=1,\ldots,N$. If the limits fit together, i.e. if
\begin{align*}
v^+_{j}= v^-_{j+1}, \quad j=1,\ldots,N-1,
\end{align*}
then one often observes $N$-waves (or multi-waves) of \eqref{b3-e2.3.1}
which look like linear superpositions  of the waves
$v_{\star,j}(x-\mu_{\star,j}t), j=1,\ldots,N$, see for example
 Figures \ref{b3-fig:QN2F.c}, \ref{b3-fig:QN2F.d} for two fronts
from example \eqref{b3-e3.1.1} travelling at different speeds to the left and superimposed onto each other. {\it Strong interaction} occurs when two or several
fronts move towards each other, while all other cases are called {\it weak interactions}. Many more interaction phenomena of this type may be found in \cite{b3-WIN16} and the references  therein.
\index{interaction!strong} \index{interaction!weak}

In \cite{b3-BOR14,b3-S09} we extend the freezing method in order to handle
such interactions. More precisely, we generalise \eqref{b3-e2.2.1} to
\begin{equation} \label{b3-e2.3.2}
u(x,t)=  \sum_{j=1}^{N} v_j(x-\gamma_j(t),t),
\end{equation}
where the values of $\gamma_{j}:\RR \rightarrow \RR$ denote the time-dependent
position of the $j$-th profile $v_j:\RR \rightarrow \RR^m$ which we expect
to have limits 
\begin{equation*}
\lim_{\xi \rightarrow \pm\infty}v_{j}(\xi,t) =v_j^{\pm}- w_j^-, \quad
w_j^-=  \begin{cases} 0, & j=1,\\ v_j^- & j \ge 2. \end{cases}
                    \end{equation*}
The main idea is to combine \eqref{b3-e2.3.2} with
a dynamic partition of unity
\begin{equation*}Q_j(\gamma(t),x) = \frac{\varphi(x-\gamma_j(t))}
{\sum_{k=1}^N \varphi(x - \gamma_k(t))}, \quad j=1,\ldots, N,
\end{equation*}
where $\varphi\in C^{\infty}(\RR,(0,1])$ is a mollifier function, for example
$\varphi(x) =\mathrm{sech}(\beta x)$ for some $\beta >0$.
Using \eqref{b3-e2.3.2} in \eqref{b3-e2.3.1} and abbreviating
$v_k(\star)=v_k(\cdot - \gamma_k(t),t)$, one finds
\begin{equation*} \label{b3-e2.3.3}
\begin{aligned}
\sum_{j=1}^N& \left[v_{j,t}(\star)-v_{j,\xi}(\star) \gamma_{t,j}\right] = u_t =
Au_{xx}+ f(u) \\
= & \sum_{j=1}^N \Big[ Av_{j,\xi \xi}(\star) + f(v_j(\star)+w_j^-) \Big.\\
+& \Big. Q_j(\gamma,\cdot) \Big\{ f\Big( \sum_{k=1}^N v_k(\star)\Big) -
\sum_{k=1}^N f(v_k(\star)+w_k^-) \Big\} \Big].
\end{aligned}
\end{equation*}
Equating the terms inside brackets $\big[ \cdot \big]$ on both sides,
substituting $\xi=x- \gamma_j(t)$ and adding phase conditions and initial conditions
leads to the following {\it decompose and freeze} system (see \cite{b3-BOR14,b3-S09})
\begin{equation} \label{b3-e2.3.4}
\begin{aligned}
  v_{j,t}(\xi,t)= &  Av_{j,\xi\xi}(\xi,t) + v_{j,\xi}(\xi,t)\mu_j(t) + f(v_j(\xi,t)) \\
 + & \frac{\varphi(\xi)}{\sum_{k=1}^{N}\varphi(*_{kj})}
  \Big[f\Big(\sum_{k=1}^{N}v_k(*_{kj},t)\Big)-\sum_{k=1}^{N}f(v_k(*_{kj},t)+w_k^-)
  \Big], \\
  0 = &  \left(v_j(\cdot,t)-\hat{v}_j,\hat{v}_{j,\xi}\right)_{L^2}
  \text{, }\quad v_{j}(\cdot,0)=v_j^{0}, \\
  \gamma_{j,t}=  & = \mu_j\text{, }\qquad\qquad\qquad\quad\quad\;\;\, \gamma_j(0) = \gamma_{j}^{0}. 
\end{aligned}
\end{equation}
This is an $N$-dimensional PDAE to be solved for $(v_j,\mu_j,\gamma_j)$, $j=1,\ldots,N$, where 
\begin{equation*}
  \label{b3-equ:s4.1.14}
  *_{kj}=\xi-\gamma_k(t)+\gamma_j(t),\quad
  \varphi\in C^{\infty}(\mathbb{R},(0,1]),\quad  u_0 = \sum_{j=1}^{N}v_j^0(\cdot-\gamma_j^0).
\end{equation*}
A particular difficulty of this system is that the right-hand side contains
non-local terms $v_k(\star_{kj},t)$ which need special treatment when
discretised on bounded intervals $[x_-,x_+]$, see Section \ref{b3-s4.6}.
We also mention that the stability of this approach for weak interaction
is analysed in \cite{b3-BOR14,b3-S09} and that a generalisation of the decompose
and freeze method to the abstract
framework of Section \ref{b3-s2.1} is proposed and applied in
\cite{b3-BOR14,b3-BST08,b3-O14}, see also Section \ref{b3-s4.6}.




\sect[Applications]{Applications to Parabolic, Hyperbolic,  and
Hamiltonian Systems}
\label{b3-s4}
\subsection[Parabolic Systems]{Travelling and Rotating Waves in Parabolic  Systems}
\label{b3-s4.1}

Our first numerical example deals with a scalar parabolic equation \eqref{b3-e2.3.1}
related to the classical Nagumo equation with a cubic non-linearity.
\begin{example}[Quintic Nagumo equation]\label{b3-exa:s4.1.1}
In the scalar case $m=1$ with ${A=1}$,
\begin{align}
  \label{b3-e3.1.1}
  f(u,u_x) = -\prod_{i=1}^{5}(u-b_i),\quad b_i\in\mathbb{R},\quad 0=b_1<b_2<b_3<b_4<b_5=1.
\end{align}
Equation \eqref{b3-e3.1.1} is called the quintic Nagumo equation (short: QNE), \cite{b3-BOR14}. 

\begin{figure}[ht]
  \centering
  \subfigure[]{\includegraphics[width=0.25\textwidth] {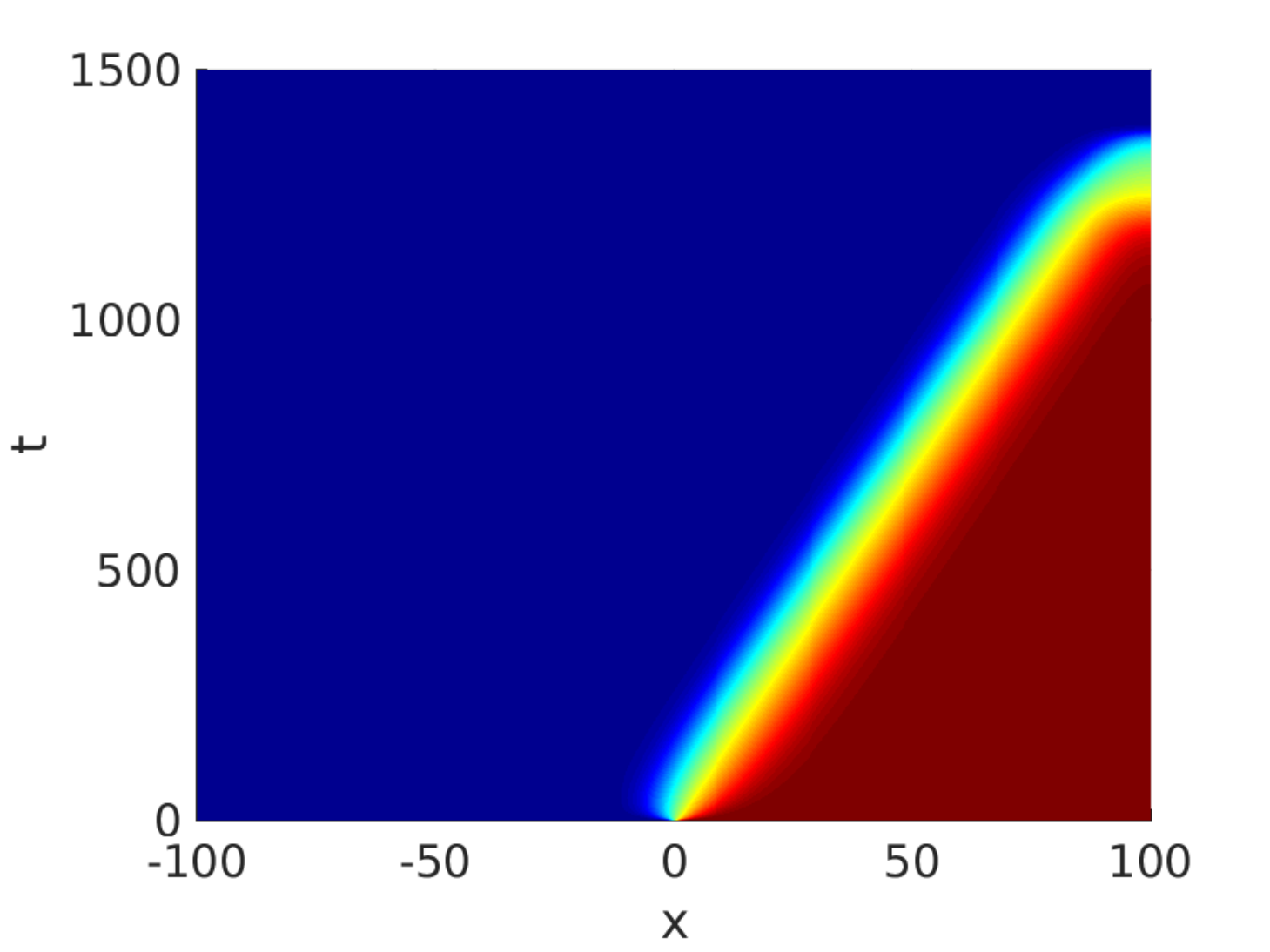}\label{b3-fig:QN.a}}
  \hspace*{-0.45cm}
  \subfigure[]{\includegraphics[width=0.25\textwidth] {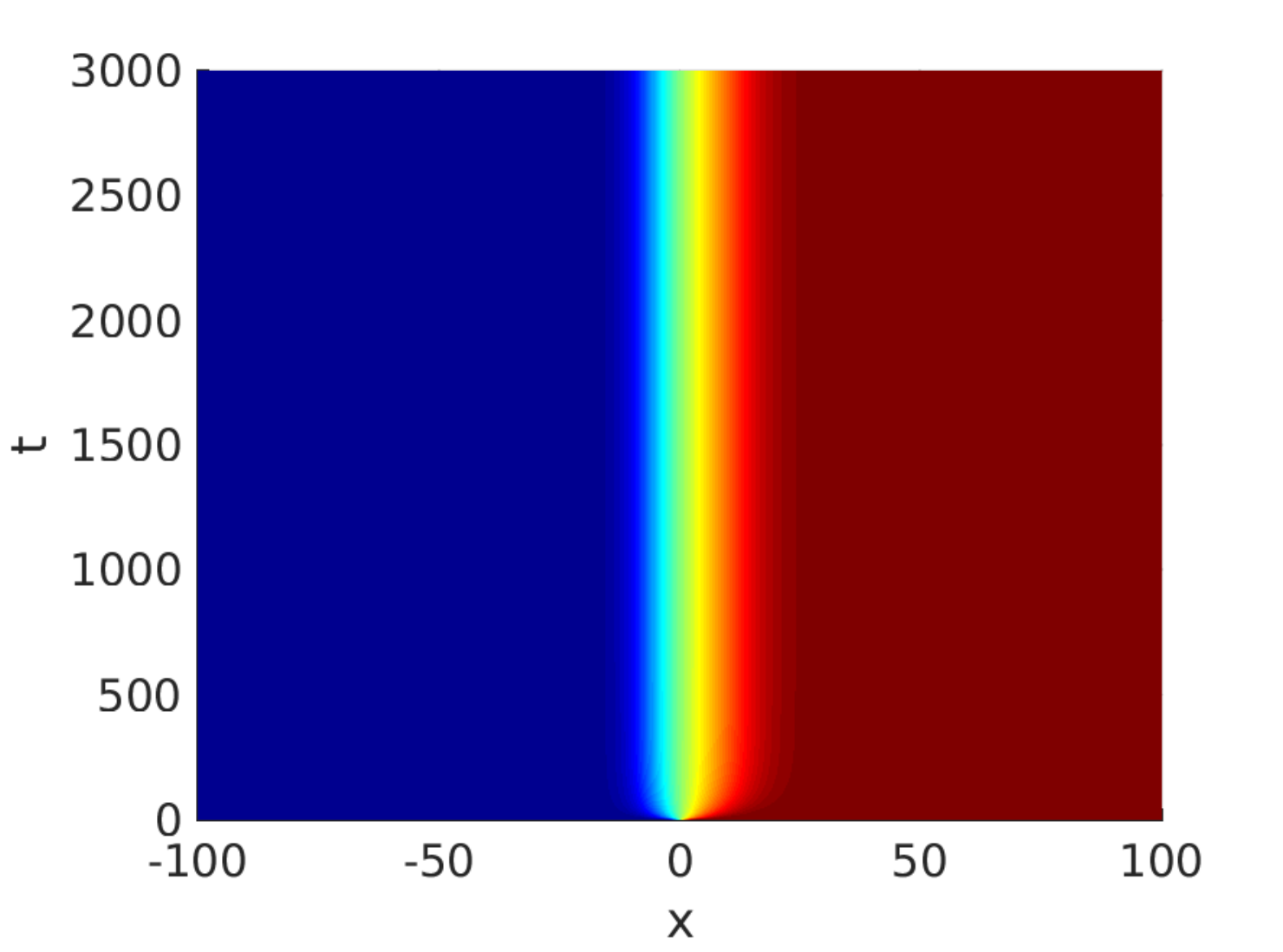} \label{b3-fig:QN.b}}
  \hspace*{-0.45cm}
  \subfigure[]{\includegraphics[width=0.25\textwidth] {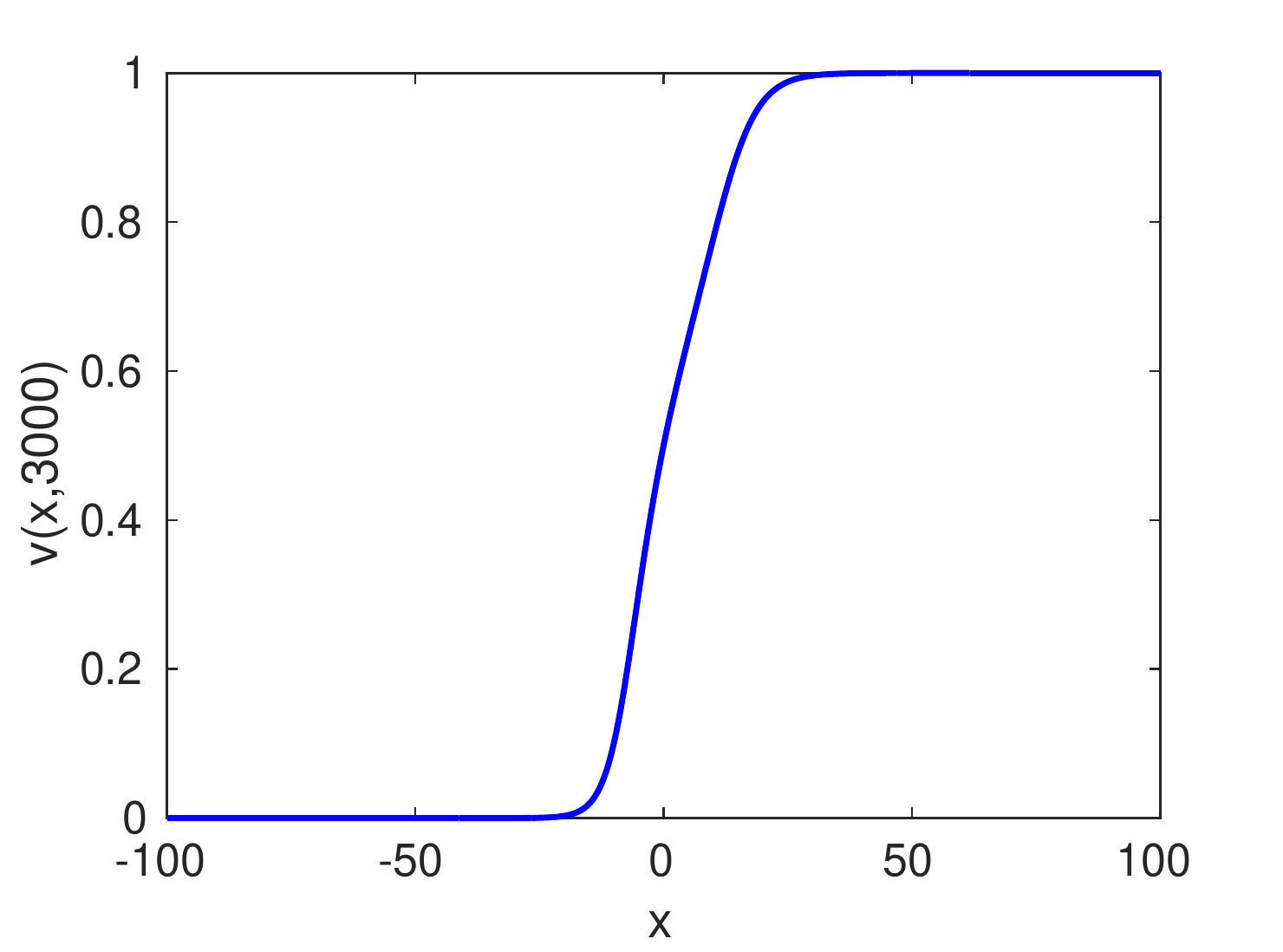}\label{b3-fig:QN.c}}
  \hspace*{-0.45cm}
  \subfigure[]{\includegraphics[width=0.25\textwidth] {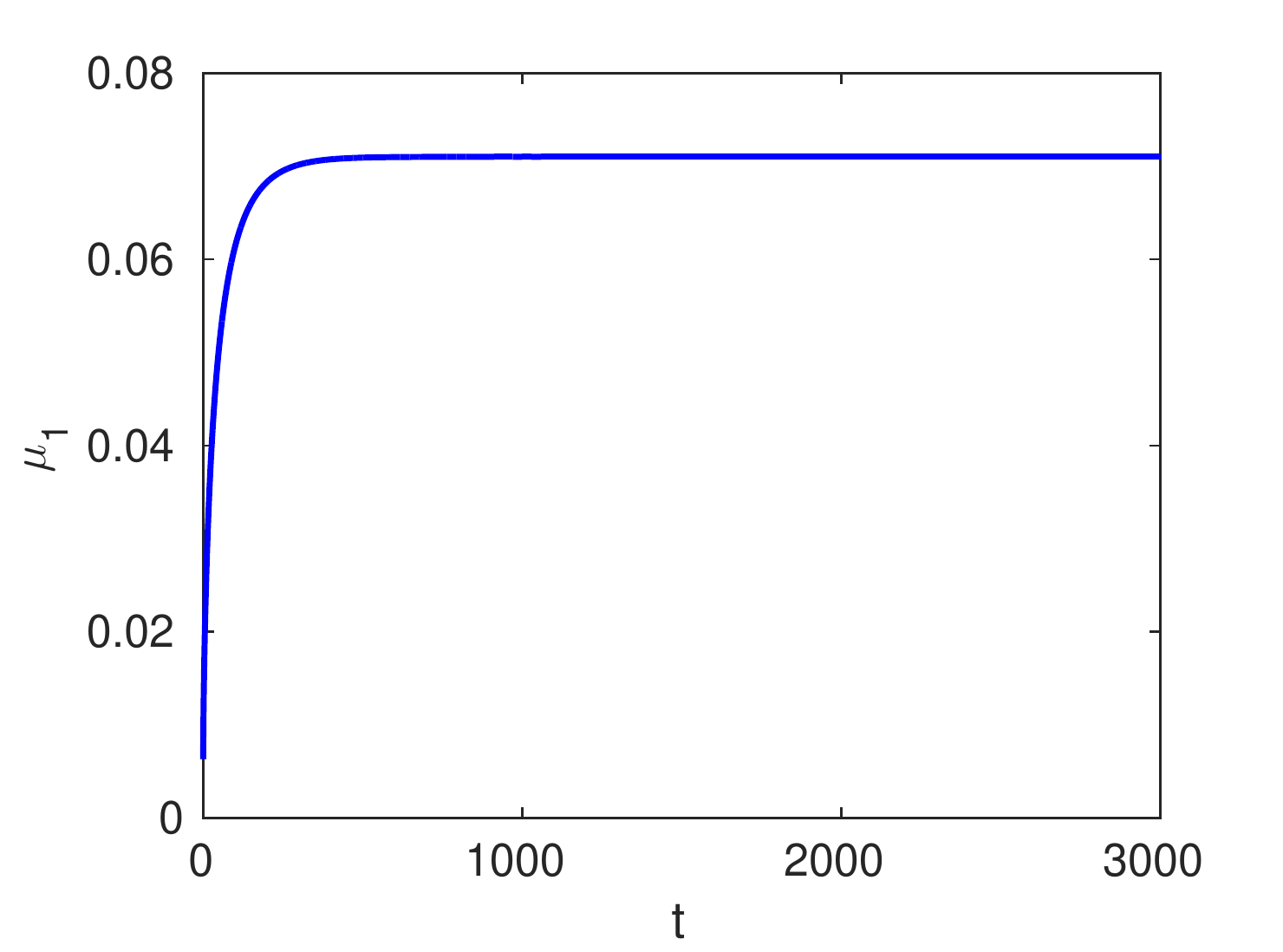}\label{b3-fig:QN.d}}
  \caption{QN-front: space-time of $u$ (a), of $v$ (b), profile $v$ (c), velocity $\mu$ (d)}
  \label{b3-fig:QN}
\end{figure}
Figure \ref{b3-fig:QN.a} shows the time evolution for a travelling front of the QNE for parameters $b_2=\tfrac{2}{5}$, $b_3=\tfrac{1}{2}$, $b_4=\tfrac{17}{20}$, 
spatial domain $[-100,100]$, initial data $u_0(x)=\tfrac{\tanh(x)+1}{2}$ and time range $[0,1500]$. At time $t\approx 1300$ the front leaves the computational 
domain. Figures \ref{b3-fig:QN.b} and \ref{b3-fig:QN.d} show the time
evolution of the front profile and the velocity obtained by solving
the freezing system
\eqref{b3-equ:s4.1.3} 
with homogeneous Neumann boundary conditions, $f$ from \eqref{b3-e3.1.1}, parameters $b_j$ and spatial domain as before, initial data $v_0=u_0$, and the  
template $\hat{v}=u_0$ on the time range $[0,3000]$. The front quickly
stabilises at the shape $v_{\star}$ shown in Figure \ref{b3-fig:QN.c},
and  the velocity quickly reaches its asymptotic value $\mu_{\star}\approx 0.07$ 
as shown in \ref{b3-fig:QN.d}. For the numerical
solution of \eqref{b3-e9}  resp.~\eqref{b3-equ:s4.1.3} we used a FEM 
space discretisation with Lagrange $C^0$-elements and maximal element size
$\triangle x=0.3$, the BDF method for time discretisation with maximum order
$2$, time step-size $\triangle t=0.3$, relative tolerances $10^{-2}$ and
$10^{-3}$, and absolute tolerances $10^{-3}$ and $10^{-4}$, combined with 
Newton's method for non-linear equations.
\end{example}
The next example is a two-dimensional system of type \eqref{b3-e10}
obtained by writing a scalar complex equation as a real system.
\begin{example}[Quintic-cubic Ginzburg-Landau equation]\label{b3-exa:s4.1.2} 
Consider the PDE
\begin{align*}
  z_t=\alpha\triangle z + g(z),\;z=z(x,t)\in\mathbb{C},\; g(z)=(\delta+\beta|z|^2+\gamma|z|^4)z,\; \alpha,\beta,\gamma\in\mathbb{C},\;\delta\in\mathbb{R},
\end{align*}
known as the quintic-cubic Ginzburg-Landau equation (short: QCGL).
\index{Ginzburg-Landau equation!quintic-cubic}

\begin{figure}[ht]
  \centering
  \hspace*{-1.3cm}
  \subfigure[]{\includegraphics[height=2.7cm]{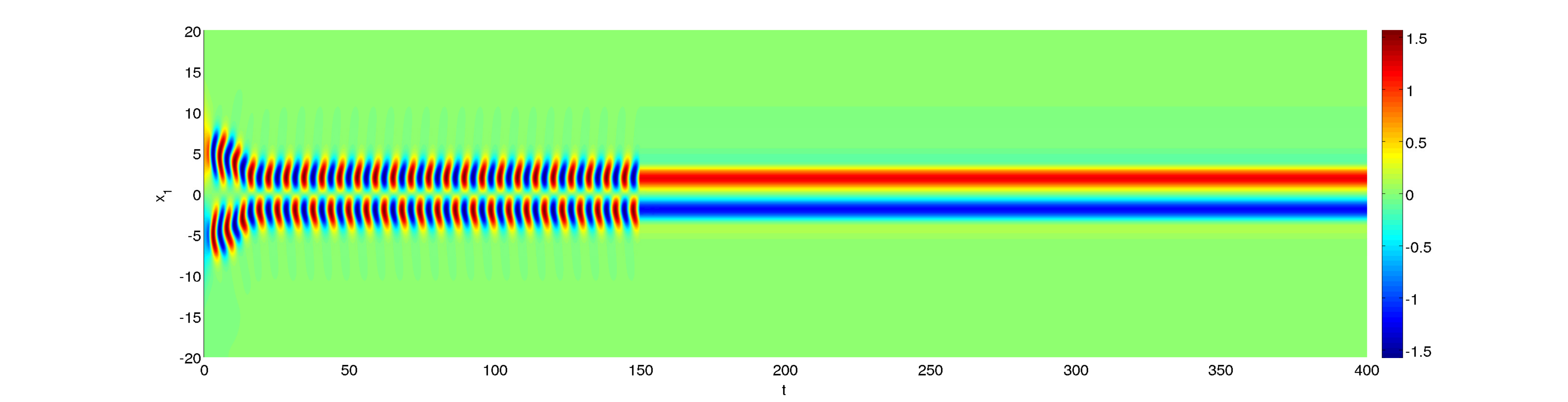}  \label{b3-fig:SS.a}}\
  \hspace*{-1.0cm}
  \subfigure[]{\includegraphics[height=2.5cm]{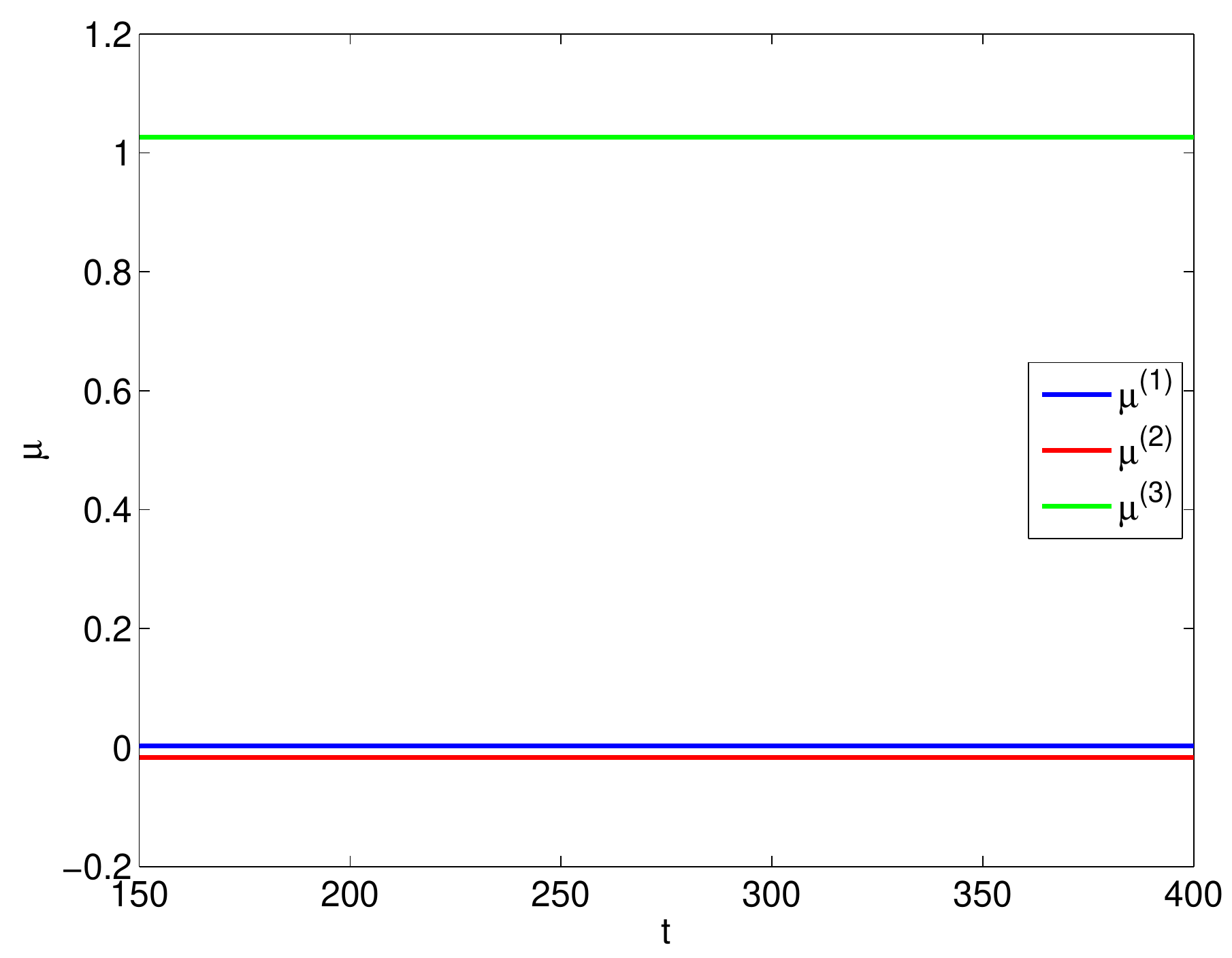}  \label{b3-fig:SS.b}}
  \caption{Spinning soliton of QCGL: space-time along $x_2=0$ (a), velocities (b)}
  \label{b3-fig:SS}
\end{figure}

Figure \ref{b3-fig:SS.a} shows the time evolution for the real part of a spinning soliton (cross-section at $x_2=0$) 
of the QCGL for parameters $\alpha=\tfrac{1}{2}+\tfrac{1}{2}i$, $\beta=\tfrac{5}{2}+i$, $\gamma=-1-\tfrac{1}{10}i$, $\delta=-\tfrac{1}{2}$,
spatial domain $B_{20}(0)=\{x\in\mathbb{R}^2:|x|\leqslant 20\}$, initial data 
$u_0(x)=(\mathrm{Re}z_0,\mathrm{Im}z_0)^{\top}$ for $z_0(x)=\tfrac{x}{5}\exp\left(-\frac{1}{49}|x|^2\right)$ and 
time range $[0,150]$. At time $t=150$ we take the solution data and switch on the freezing system \eqref{b3-equ:s4.1.7}.  
Figures \ref{b3-fig:SS.a} and \ref{b3-fig:SS.b} show the time evolution of
the real part of the soliton profile and the velocities 
obtained by solving \eqref{b3-equ:s4.1.7} with homogeneous Neumann boundary conditions, 
parameters $\alpha,\beta,\gamma,\delta$ and spatial domain as before, initial data $v_0=u(\cdot,150)$, template function $\hat{v}=u(\cdot,150)$ 
and time range $[150,400]$. Approximations of the real part of the soliton profile $v_{\star}$ and the velocities 
$\mu_{\star}=\left(\begin{smallmatrix}S_{\star}&a_{\star}\\0&0\end{smallmatrix}\right)$ with $S_{\star}\approx\left(\begin{smallmatrix}0&1.027\\-1.027&0\end{smallmatrix}\right)$ 
and $c_{\star}\approx\left(\begin{smallmatrix}0.003\\-0.017\end{smallmatrix}\right)$, are shown in Figures \ref{b3-fig:RW.a} and \ref{b3-fig:SS.b}. 
For the numerical solution of \eqref{b3-e10} resp.~\eqref{b3-equ:s4.1.7}
 we used FEM for space discretisation 
with Lagrange $C^0$-elements and maximal element size $\triangle x=0.25$,
the BDF method for time discretisation 
with maximum order $2$, time step-size $\triangle t=0.1$ resp.~$\triangle t=0.2$
, 
relative tolerance $10^{-4}$ resp.~$10^{-2}$, and absolute tolerance $10^{-5}$
resp.~$10^{-7}$, and 
Newton's method for non-linear systems.
\end{example}
\begin{figure}[ht]
  \centering
  \subfigure[]{\includegraphics[width=0.235\textwidth]{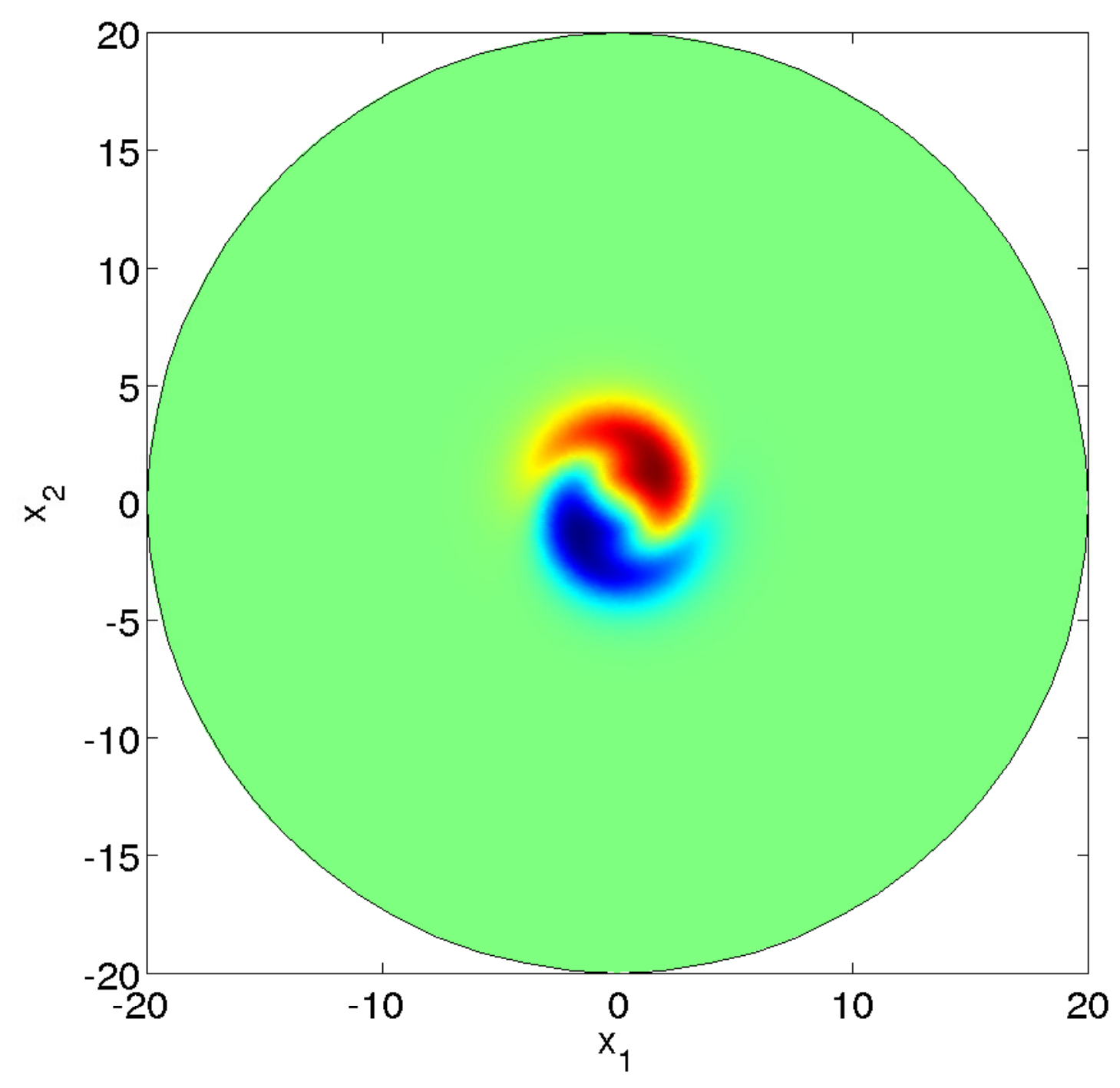}  \label{b3-fig:RW.a}}
  \hspace*{-0.25cm}
  \subfigure[]{\includegraphics[width=0.235\textwidth]{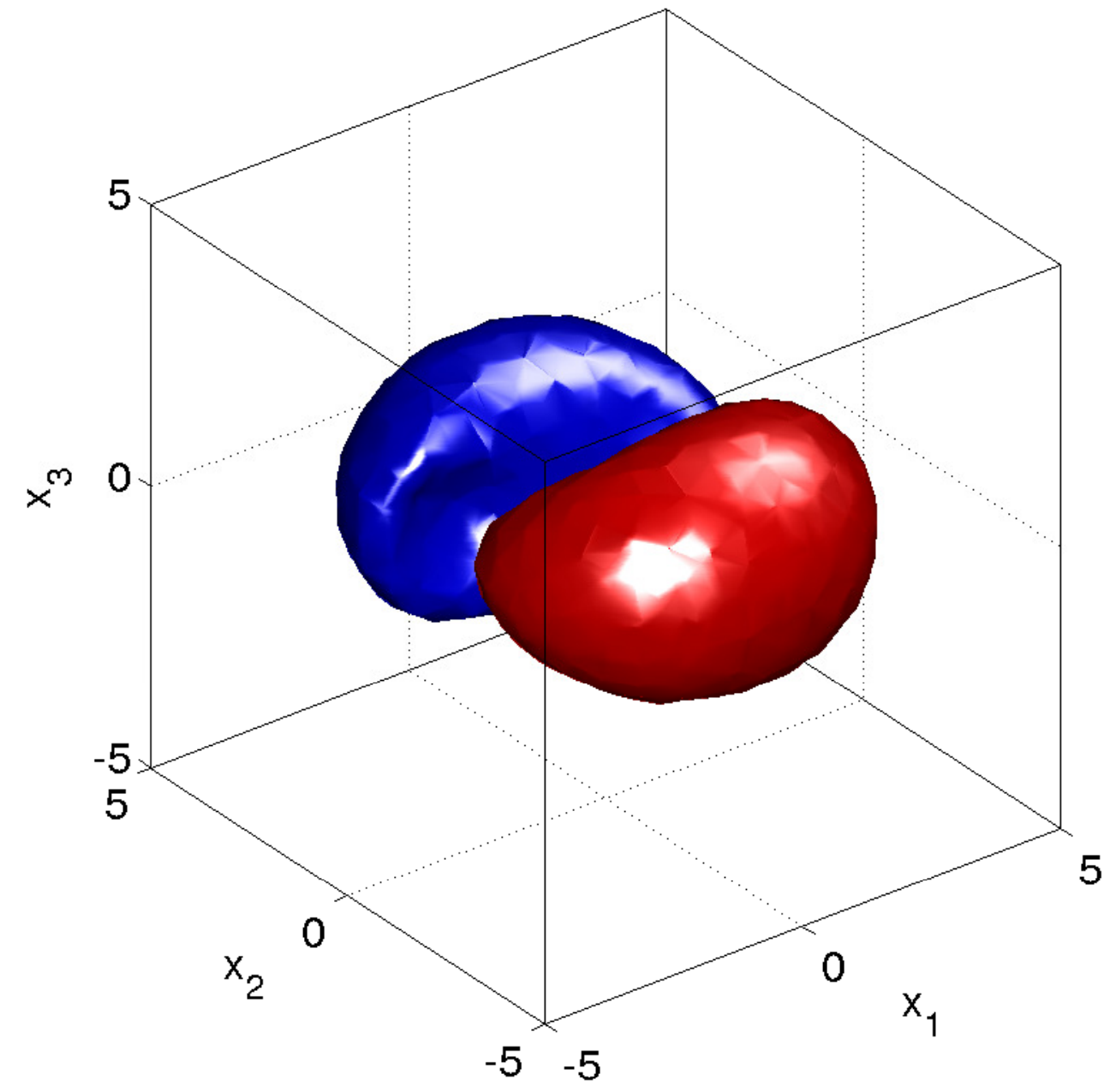}  \label{b3-fig:RW.b}}
  \hspace*{-0.25cm}
  \subfigure[]{\includegraphics[width=0.235\textwidth]{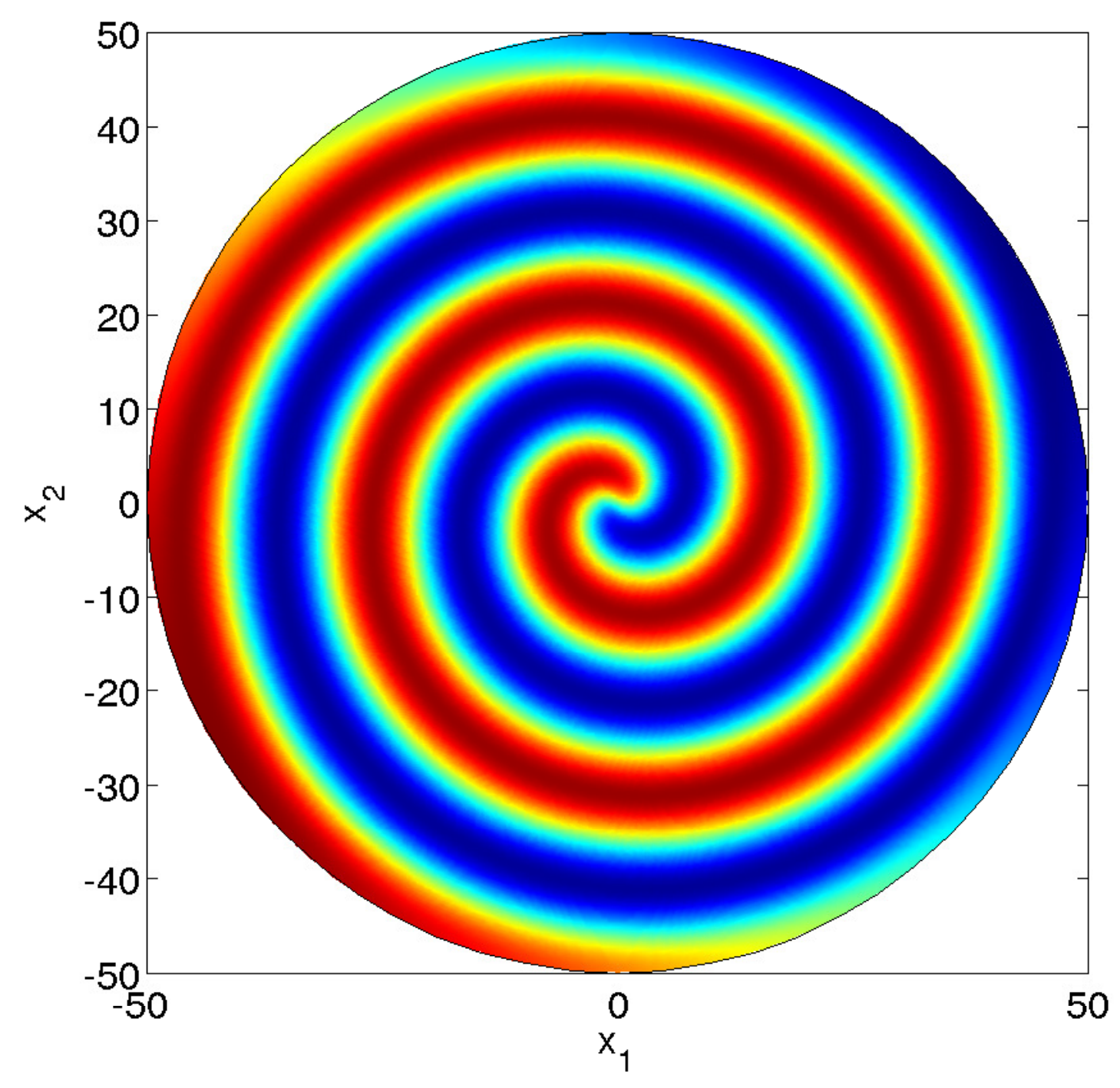}         \label{b3-fig:RW.c}}
  \hspace*{-0.25cm}
  \subfigure[]{\includegraphics[width=0.26\textwidth]{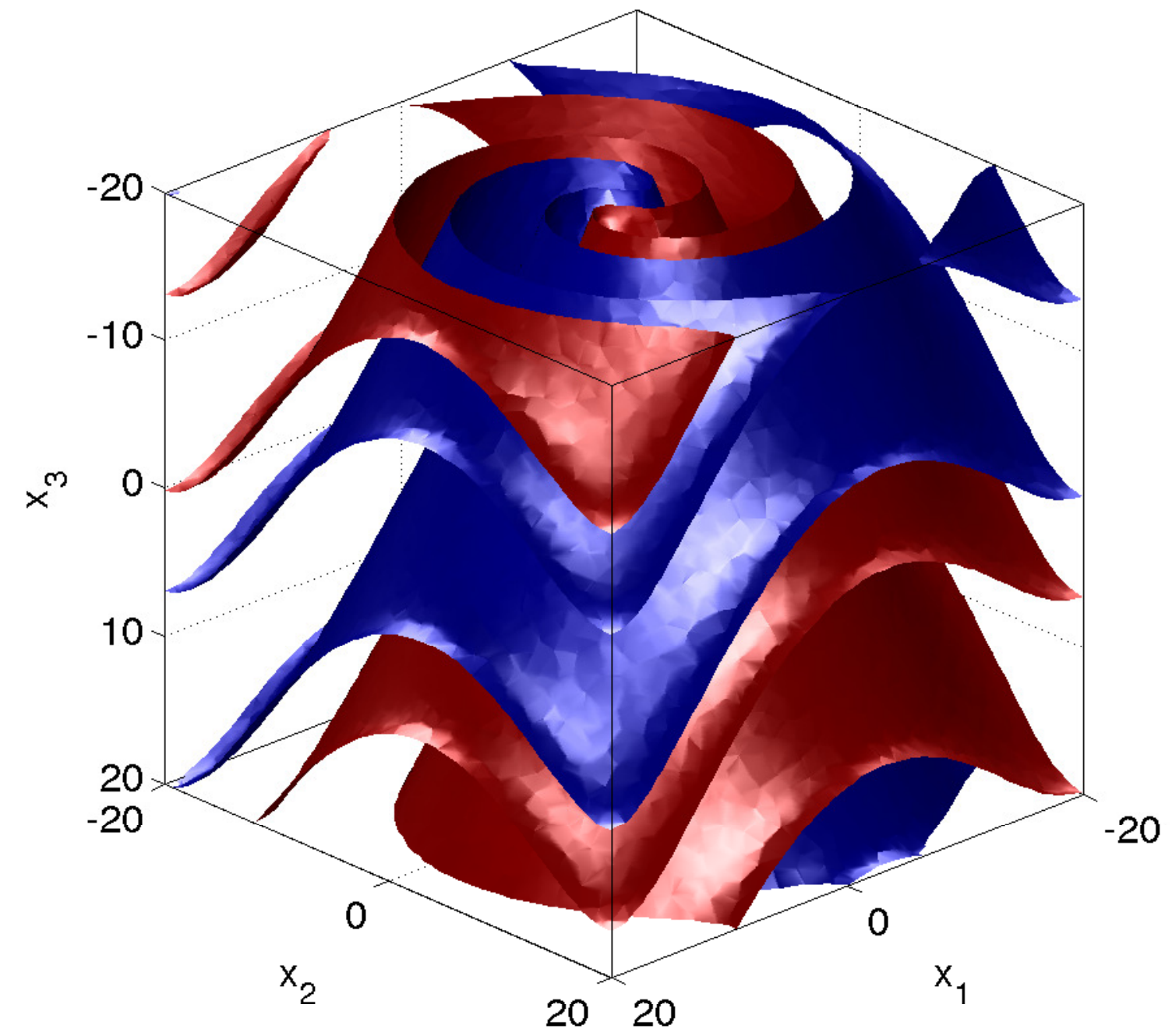} \label{b3-fig:RW.d}}
  \caption{Rotating waves: spinning soliton for $d=2$ (a) and $d=3$ (b), spiral wave for $d=2$ (c), untwisted scroll wave for $d=3$ (d)}
  \label{b3-fig:RW}
\end{figure}

The spinning solitons of the QCGL from Example~\ref{b3-exa:s4.1.2} are a special kind of a localised rotating wave for $d=2$, 
see Fig.~\ref{b3-fig:RW.a}. Their extension to $d=3$ dimensions is
displayed in  Fig.~\ref{b3-fig:RW.b}, and non-localised rotating waves, 
such as spiral waves are shown in Fig.~\ref{b3-fig:RW.c}. Finally, we
show a so-called scroll wave in Fig.~\ref{b3-fig:RW.d}. These types of waves
occur in 
various applications, e.g. in the QCGL \cite{b3-CMM01,b3-MMCML00}, the $\lambda-\omega$-system \cite{b3-KK81}, 
the Barkley model \cite{b3-B94}, and the FitzHugh-Nagumo system
\cite{b3-FH61}. Their treatment via the freezing method is discussed in
the papers \cite{b3-O14,b3-BOR14,b3-BL08}.

\subsection{Hyperbolic Systems}
\label{b3-s4.2}
The following hyperbolic system in one space dimension may be viewed as a
special case of \eqref{b3-e9} with $A=0$,
\begin{align}
  \label{b3-e4.2.1}
  u_t = Eu_x + f(u), &&u(\cdot,0)=u_0.
\end{align}
For \eqref{b3-e4.2.1} to be well-posed, we assume   $E\in\mathbb{R}^{m,m}$ to
be real diagonalisable and $f:\mathbb{R}^{m}\rightarrow\mathbb{R}^m$ to be sufficiently smooth. As in Section \ref{b3-s1.3}
travelling waves of \eqref{b3-e4.2.1} are solutions of the form
\eqref{b3-e10a}, the underlying Lie group $(G,\circ)$ is the additive
group $(\mathbb{R},+)$ acting on functions via translations.
The freezing system \eqref{b3-equ:s4.1.3} for pursuing profiles and velocities
now reads for the unknown quantities $(v,\mu,\gamma)$ as follows,
\begin{equation*} \label{b3-e4.2.2}
\begin{aligned}
  v_t =& Ev_{\xi} + \mu v_{\xi} + f(v), \quad \quad v(\cdot,0)=v_0,\\
  0 =& \langle \hat{v}_{\xi},v-\hat{v}\rangle_{L^2(\mathbb{R},\mathbb{R}^m)}, \\
  \gamma_t =& \mu, \quad \quad \gamma(0)=0.
\end{aligned}
\end{equation*}
The main difference to the parabolic case \eqref{b3-equ:s4.1.3} is due to
the fact that the unknown function $\mu(t)$ of this PDAE now appears in
the principal part of the spatial operator. This creates serious
difficulties, both for the numerical and the theoretical analysis.
These have been successfully treated in the works \cite{b3-R10,b3-R12a},
and a  series of numerical examples appears in \cite{b3-R10,b3-R12a,b3-BOR14}.
Moreover, with a slightly generalised notion of equivariance
(see \cite{b3-RKML03,b3-R10})
the freezing approach has found interesting applications to detecting
similarity solutions in Burgers' equation, see \cite{b3-RKML03,b3-BOR14}
for the one-dimensional and \cite{b3-R16a,b3-R16b} for the multi-dimensional
case.

Finally, we refer to the papers \cite{b3-R12b,b3-R12c} in which 
the stability of travelling waves and the freezing approach is analysed for
mixed parabolic-hyperbolic systems of the partitioned form
\begin{align}
  \label{b3-e4.2.3}
  u_t = \begin{pmatrix}A_{11}&0\\0&0\end{pmatrix}u_{xx} + \begin{pmatrix}g(u)\\B_{22}u_2\end{pmatrix}_x + \begin{pmatrix}f_1(u)\\f_2(u)\end{pmatrix}, &&u(\cdot,0)=u_0,
  \end{align}
with a positive diagonalisable matrix $A_{11}$ and a real diagonalisable
matrix $B_{22}$. This covers the famous Hodgkin-Huxley model for propagation
of pulses in nerve axons, cf. \cite[Ch.3.1]{b3-BOR14}.

\subsection{Non-Linear Wave Equations}
\label{b3-s4.3}
Another area of application are systems of non-linear wave equations in one space dimension
\begin{align}
  \label{b3-e4.3.1}
  Mu_{tt} = Au_{xx} + \tilde{f}(u,u_x,u_t), &&u(\cdot,0)=u_0,\quad u_t(\cdot,0)=v_0,
\end{align}
where $M\in\mathbb{R}^{m,m}$ is invertible, $A\in\mathbb{R}^{m,m}$,
$\tilde{f}:\mathbb{R}^{3m}\rightarrow\mathbb{R}^m$ is smooth 
and $u_0,v_0:\mathbb{R}\rightarrow\mathbb{R}^m$ denote the initial data.
Further we assume $M^{-1}A$ to be positive diagonalisable which implies local
well-posedness of \eqref{b3-e4.3.1}. In case $m=1$, 
travelling waves \eqref{b3-e10a} for equation \eqref{b3-e4.3.1} and
their global stability have been treated in \cite{b3-GR97,b3-GJ09}. The freezing
 ansatz \eqref{b3-e2.2.1} now requires to
 solve the following second order PDAE (cf. \cite{b3-BOR16a,b3-BOR16b})
\begin{equation} \label{b3-e4.3.2}
\begin{aligned}
  M v_{tt} = & (A-\mu_1^2 M)v_{\xi\xi} + 2\mu_1 M v_{\xi t} + \mu_2 M v_{\xi} +
  \tilde{f}(v,v_{\xi},v_t-\mu_1 v_{\xi}) \\
  0 =& \langle \hat{v}_{\xi},v-\hat{v}\rangle_{L^2(\mathbb{R},\mathbb{R}^m)}, \\
  \mu_{1,t} =& \mu_2,\quad\gamma_t =\mu_1, \\
  v(\cdot,0) =& u_0,\quad v_t(\cdot,0)=v_0+\mu_1^0 u_{0,\xi},\quad \mu_1(0)=\mu_1^0,\quad \gamma(0)=0
\end{aligned}
\end{equation}
for the unknown quantities $(v,\mu_1,\mu_2,\gamma)$. Travelling waves
$(v_{\star},\mu_{\star})$ appear as steady states of \eqref{b3-e4.3.2}
(with $\mu_1=\mu_{\star}$, $\mu_2=0$) and satisfy the equation
\begin{equation*} \label{b3-e4.3.3}
  0= (A-\mu_{\star}^2 M)v_{\star,\xi\xi}  + f(v_{\star},v_{\star,\xi},-\mu_{\star} v_{\star,\xi}).
  \end{equation*}
  Differentiating the algebraic constraint in \eqref{b3-e4.3.2} w.r.t. time
  at $t=0$ and inserting the initial conditions leads to a first consistency
  condition for $\mu_1^0$
  \begin{equation} \label{b3-e4.3.3a}
  \mu_1^0 \langle u_{0,\xi},\hat{v}_{\xi} \rangle_{L^2} + \langle v_0, \hat{v}_{\xi}
  \rangle_{L^2} = 0,
  \end{equation}
  and differentiating twice at $t=0$ gives a consistency condition
  for  $\mu_2(0)=\mu_2^0$:
  \begin{equation} \label{b3-e4.3.3b}
  0 = \langle (M^{-1}A+(\mu_1^0)^2I_m) u_{0,\xi \xi} 
  + 2 \mu_1^0 v_{0,\xi} 
  +  M^{-1}f(u_0,u_{0,\xi},v_0),
  \hat{v}_{\xi} \rangle_{L^2} + \mu_2^0 \langle
  u_{0,\xi},\hat{v}_{\xi}
\rangle_{L^2}.
  \end{equation}
  The local stability of the PDAE system \eqref{b3-e4.3.2} is analysed in \cite{b3-BOR16a}
 while a generalisation to several space dimensions and a numerical example
 appear in \cite{b3-BOR16b}. 
  It is interesting to note that the system \eqref{b3-e4.3.1} may be written
  as a first order system \eqref{b3-e4.2.1} of dimension $3m$. Taking a positive
  square root $N = (M^{-1}A)^{1/2}$ and introducing
  the variables $U_1=u$, $U_2=u_t+Nu_x$, $U_3=u_t-Nu_x+c u$
  ($c\in \RR$ arbitrary) leads to a system \eqref{b3-e4.2.1} with
  \begin{equation} \label{b3-e4.3.4}
  \begin{aligned}
  E =& \begin{pmatrix} N & 0 & 0 \\ 0 & N & 0 \\ 0 & 0 & -N \end{pmatrix},
  \quad
  f(U)= \begin{pmatrix} -cU_1 + U_3 \\ g(U) \\ g(U)+c U_2 \end{pmatrix}\\
  g(U) = & M^{-1}\tilde{f}(U_1,\frac{1}{2}N^{-1}(U_2-U_3+cU_1),
  \frac{1}{2}(U_2+U_3 - c U_1)).
  \end{aligned}
  \end{equation}
  Though we prefer to solve numerically the second order system
  \eqref{b3-e4.3.2}, the first order system (with a suitable choice of the constant
  $c$) is useful for 
  applying the stability results from \cite{b3-R12a}, see
  \cite{b3-BOR16a} and Section \ref{b3-s3}.
\begin{example}[Quintic Nagumo wave equation]\label{b3-exa:s4.1.3}
Taking the quintic non-linearity $f=\tilde{f} $ from \eqref{b3-e3.1.1}
with the wave equation \eqref{b3-e4.3.1} we obtain the
quintic Nagumo wave equation (short: QNWE), see \cite{b3-BOR16b}. 
\begin{figure}[ht]
  \centering
  \subfigure[]{\includegraphics[width=0.25\textwidth] {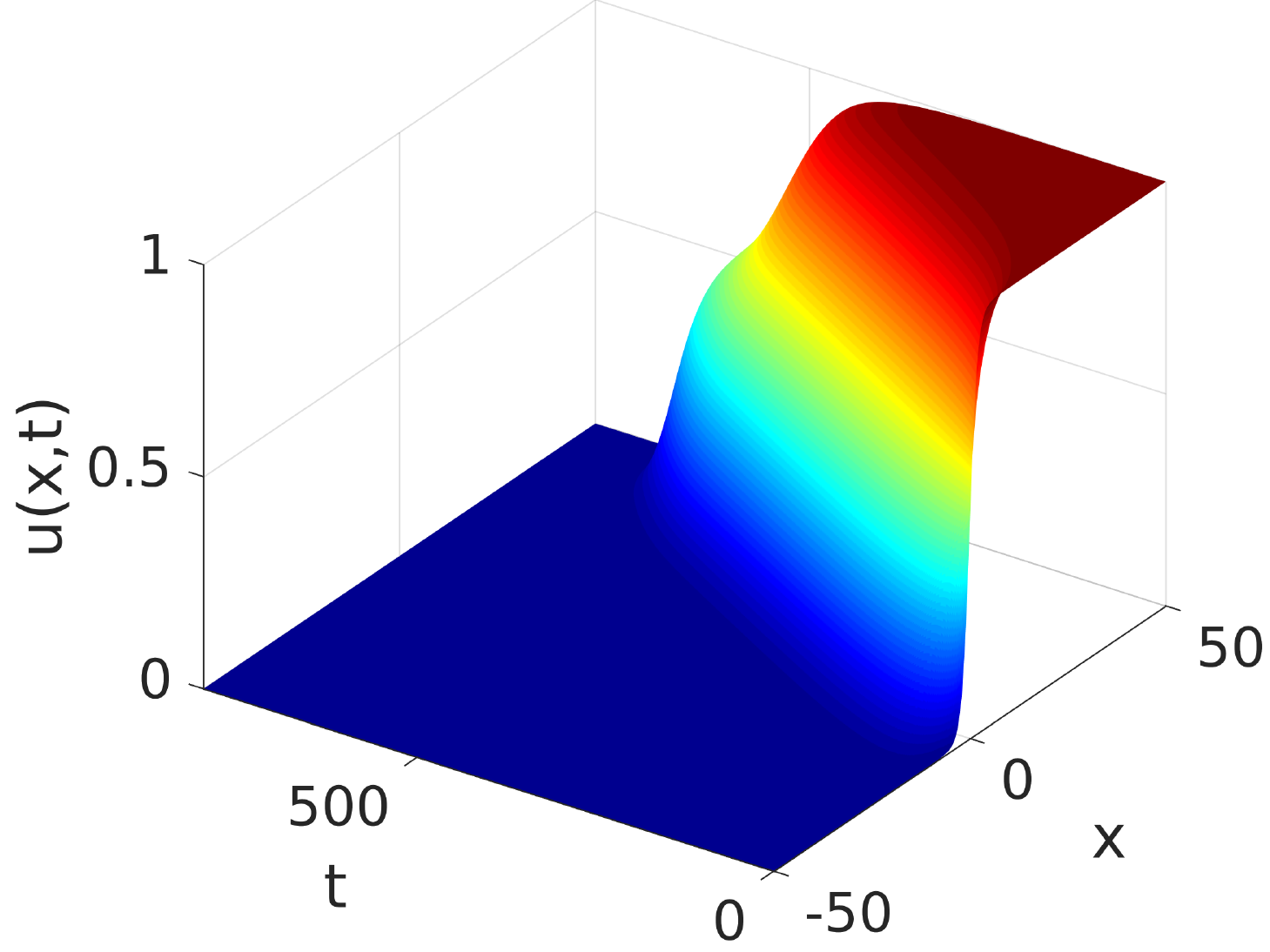}\label{b3-fig:QNWE.a}}
  \hspace*{-0.45cm}
  \subfigure[]{\includegraphics[width=0.25\textwidth] {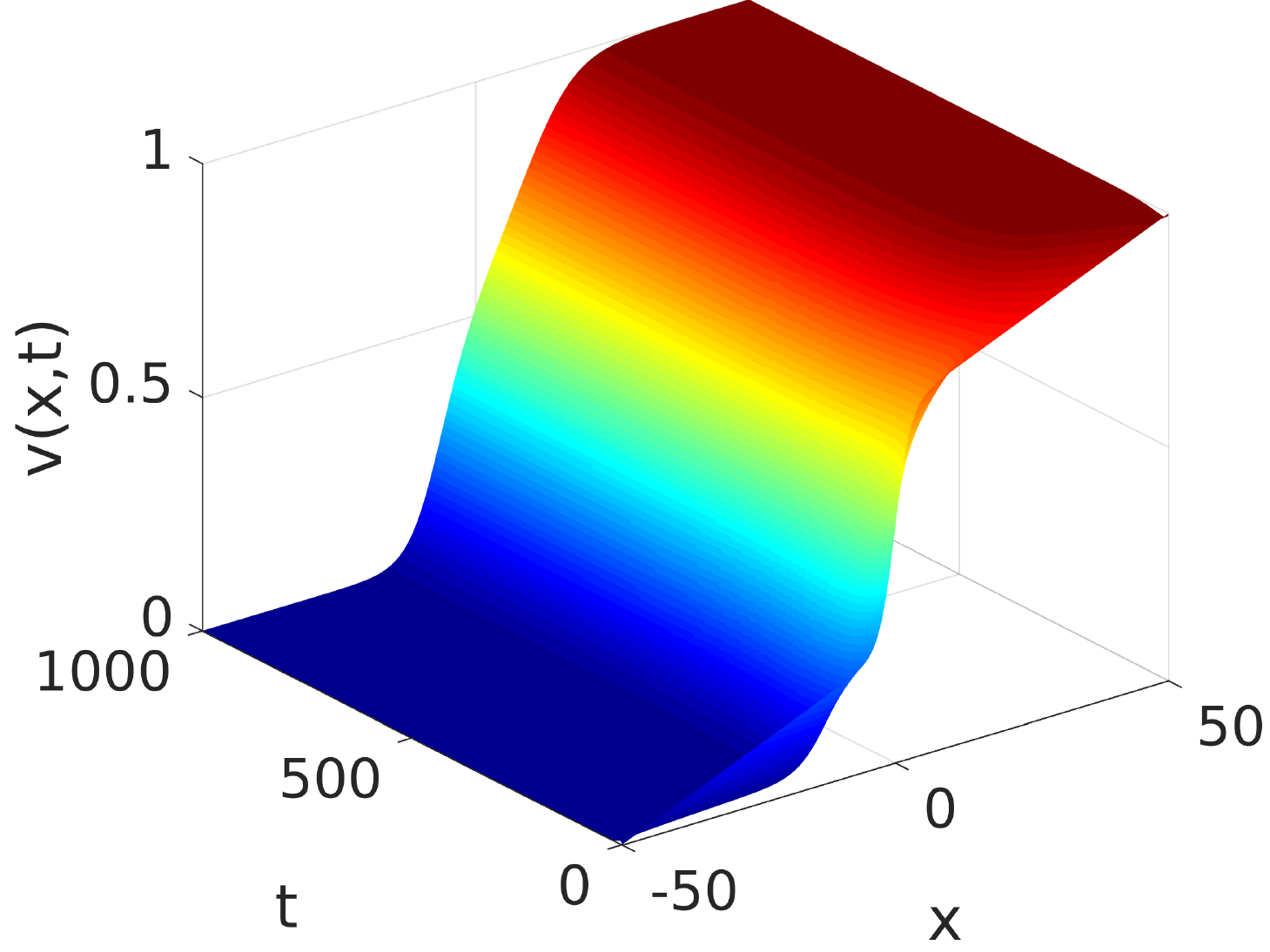} \label{b3-fig:QNWE.b}}
  \hspace*{-0.45cm}
  \subfigure[]{\includegraphics[width=0.25\textwidth] {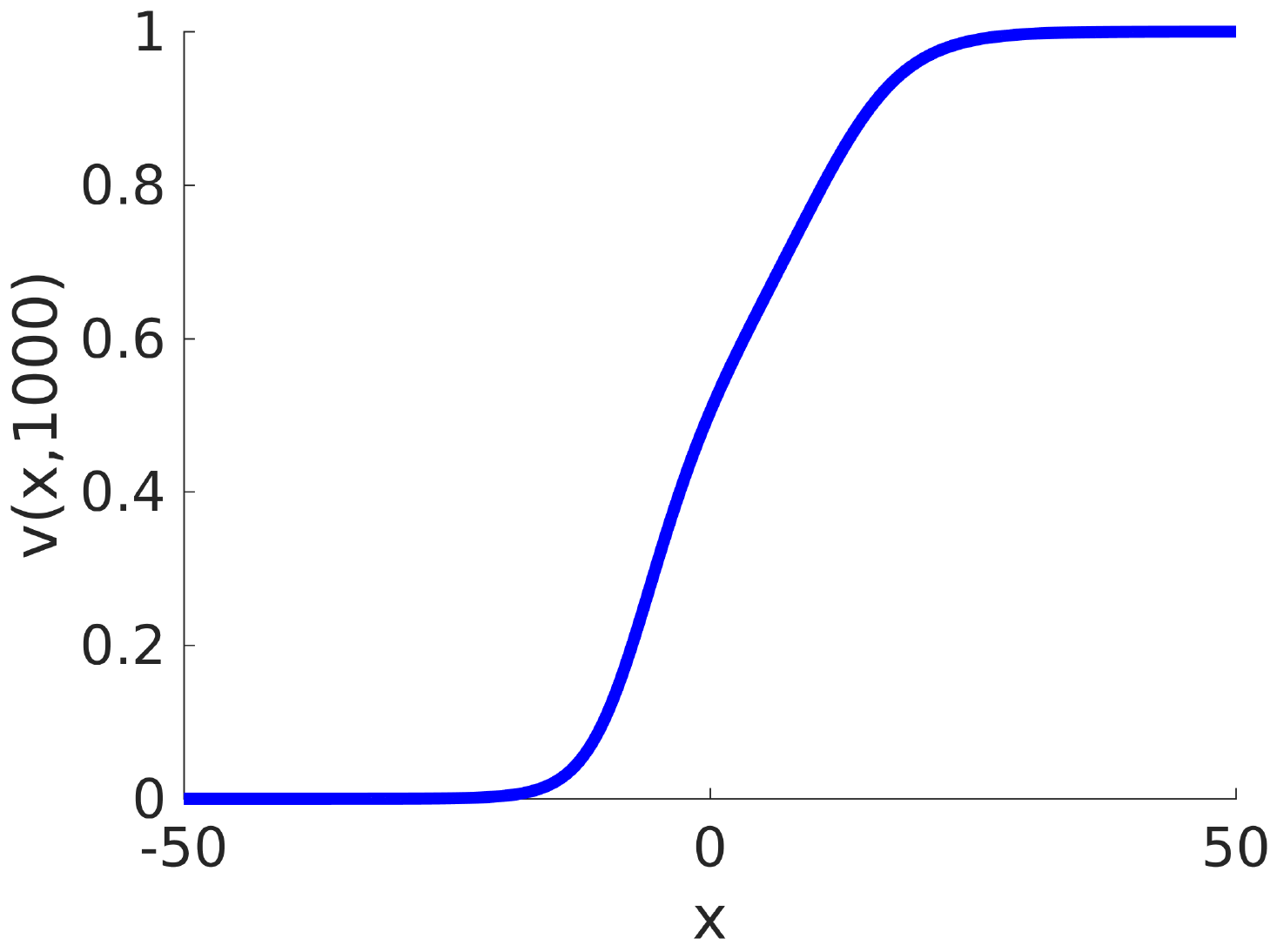}\label{b3-fig:QNWE.c}}
  \hspace*{-0.45cm}
  \subfigure[]{\includegraphics[width=0.25\textwidth] {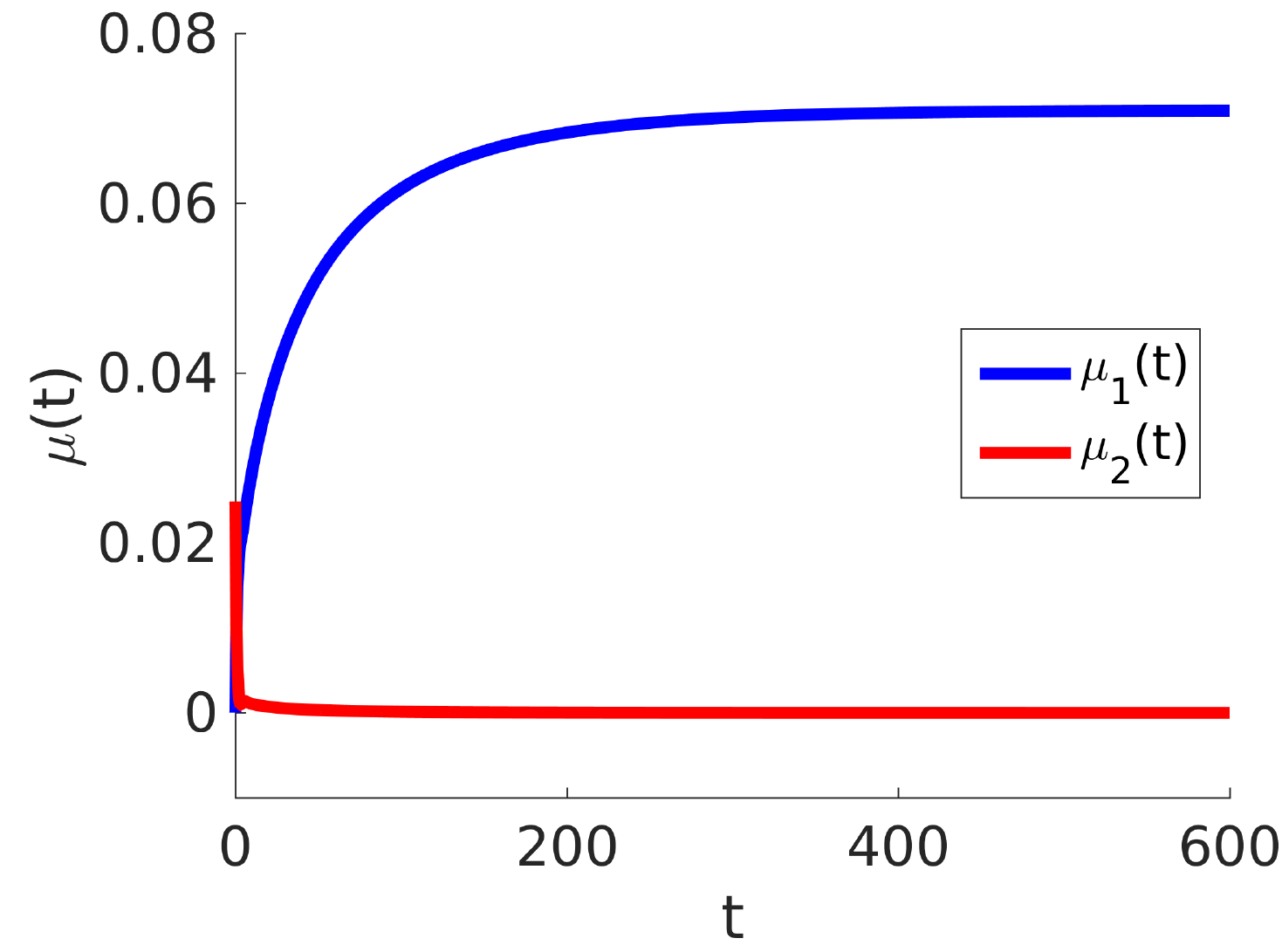}\label{b3-fig:QNWE.d}}
  \caption{QNWE-front: space-time of $u$ (a), $v$ (b), profile $v$ (c), velocity $\mu$ (d)}
  \label{b3-fig:QNWE}
\end{figure}

Figure \ref{b3-fig:QNWE.a} shows the time evolution for a travelling front of the QNWE for parameters $M=\tfrac{1}{2}$, $b_2=\tfrac{2}{5}$, $b_3=\tfrac{1}{2}$, 
$b_4=\tfrac{17}{20}$, spatial domain $[-50,50]$, initial data $u_0(x)=\tfrac{1}{2}(1+\tanh(\tfrac{x}{2}))$, $v_0(x)=0$ and time range $[0,800]$. At time $t\approx 600$ 
the front leaves our computational domain. Figures \ref{b3-fig:QNWE.b} and
\ref{b3-fig:QNWE.d} show the time evolution of the front profile and the
velocity 
obtained by solving \eqref{b3-e4.3.2} with homogeneous Neumann boundary conditions, parameters $M$, $b_j$, spatial domain and 
initial data as before, template $\hat{v}=u_0$ and time range $[0,1000]$. An
approximation of the front profile $v_{\star}$ (with $v_{-}=0$, $v_{+}=1$) and
the approach towards the limit velocity $\mu_{\star}\approx 0.07$ are shown in Figures \ref{b3-fig:QNWE.c} and \ref{b3-fig:QNWE.d}. 
The data for the numerical solution of \eqref{b3-e4.3.1}
resp.~\eqref{b3-e4.3.2} are the same as in Example \ref{b3-exa:s4.1.1},
except for the step-sizes $\triangle x=0.1$ and $\triangle t=0.2$. 
\end{example}

\subsection{Hamiltonian PDEs}
\label{b3-s4.5}
So far we mainly considered waves in dissipative PDEs which are
detected during simulation via the freezing method due to their asymptotic stability.
This changes fundamentally for PDEs with Hamiltonian structure which
typically allow several or even infinitely many conserved quantities.
They fit into the general class of
evolution problems described in Section \ref{b3-s1.1} but require quite
different techniques for establishing existence  and uniqueness of
wave solutions \cite{b3-F15} as well as their stability
(\cite{b3-GSS87,b3-GSS90}).

As a  model example consider the  cubic non-linear Schr\"{o}dinger equation
(NLS, see the references \cite{b3-C03,b3-F15,b3-KFC15,b3-SS99})
\begin{equation}
\label{b3-sd-eq01}
iu_t = - u_{xx} - |u|^2u, \quad u(\cdot,0)=u_0,
\end{equation}
which may be subsumed under \eqref{b3-e1} with $X=H^{1}(\mathbb{R};\mathbb{C})$
, $Z=H^{3}(\mathbb{R};\mathbb{C})$.
\index{Non-Linear Schr\"odinger equation}
Equivariance holds with respect to the action
\begin{equation*}
\label{b3-sd-eq02}
a(\gamma)v = e^{-i\gamma_1}v(\cdot - \gamma_2), \quad \gamma=(\gamma_1,\gamma_2) \in G
\end{equation*}
of the two-dimensional Lie group $G=S^1\times \mathbb{R}$. With $\mu=(\mu_1,\mu_2) \in \RR^2$ the freezing system \eqref{b3-e21} is given by
\begin{equation}
\label{b3-sd-eq03}
iv_t  = -v_{\xi \xi}-|v|^2v - \mu_1 v + i \mu_2 v_{\xi}, \quad v(\cdot,0)=u_0,
\end{equation}
and the fixed phase condition \eqref{b3-e24} with some $\hat{v}\in X$ reads
\begin{equation}
\label{b3-sd-eq04}
0= \big\langle i \hat v,v \big\rangle_0 = \big\langle  \hat v_{\xi},v \big\rangle_0,
\end{equation}
where
$\langle u, v \rangle_0 = \mathrm{Re}\int_{\RR}\bar{u}(x) v(x) dx $.
There is a well-known  two-parameter family of solitary wave solutions given
by 
\begin{equation} \label{b3-sd-eq05}
\begin{aligned}
u_{\star}(x,t) =& e^{i \mu_1 t}v_{\star}(x- \mu_2 t; \mu_1,\mu_2),\\
v_{\star}(\xi;\mu_1,\mu_2)=&\frac{\omega \sqrt{2} e^{i\mu_2\xi/2}}{\mathrm{cosh}(\omega \xi)},
\quad \omega^2 = \mu_1 - \frac{\mu_2^2}{4},
\end{aligned}
\end{equation}
\noindent see for example \cite[Ch.II.3]{b3-F12}.
\index{solitary waves}
For the following numerical computations we choose parameter values
$\mu_2=0.3$, $\omega=1$. 
\begin{figure} 
\centering
\includegraphics[width=0.49\textwidth]{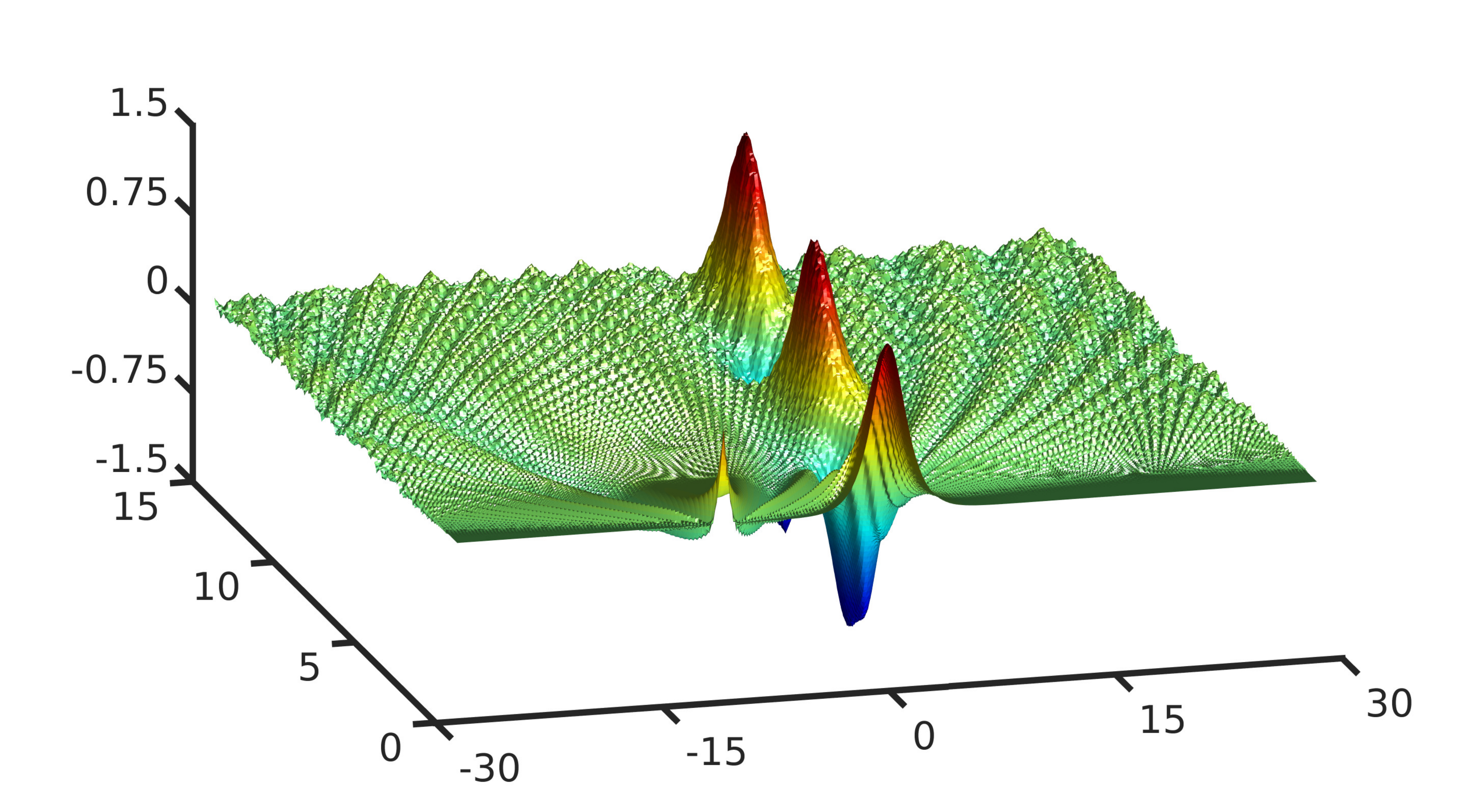}
\includegraphics[width=0.49\textwidth]{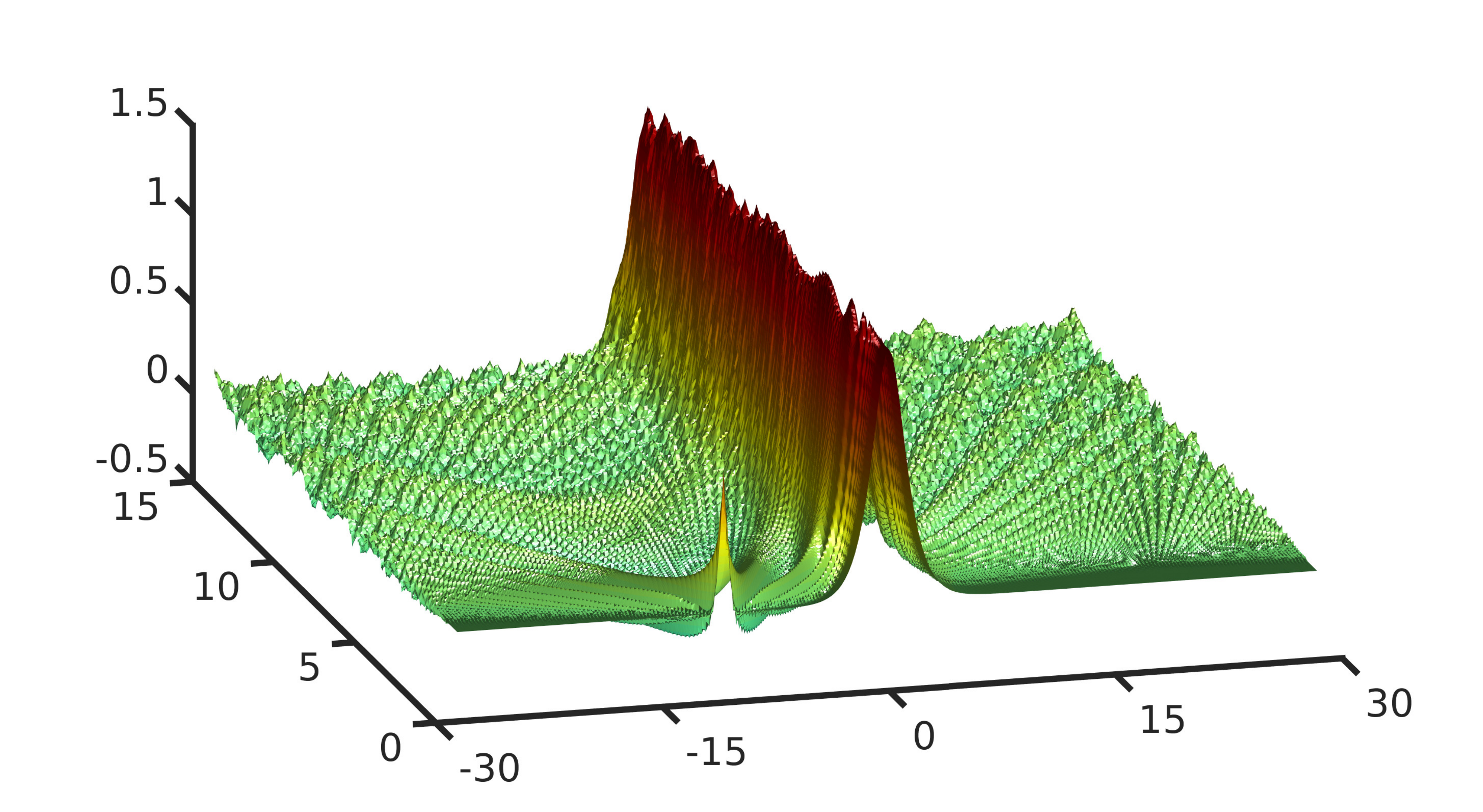}
\caption{Solitary wave of the NLS with spike-like initial perturbation:\newline
direct numerical simulation (left) vs. solution of the freezing system (right)}
\label{b3-figNLS}
\end{figure}
Discretisation in time is done via a split-step Fourier method with step
size $\Delta t = 10^{-3}$. The spatial grid is formed by $2K=256$ equidistant
points on the interval $[x_-,x_+]$ with
$x_+=-x_-=\tfrac{\pi}{0.11}\approx 28.56$. A spike-like perturbation
at $x=-11$ is added to the initial data. Figure \ref{b3-figNLS}
shows the solution for both the original system and the freezing system.
Clearly, the freezing system prevents the wave from rotating and travelling,
while the  interference patterns caused by the initial perturbation
are essentially preserved.
A theoretical result supporting these observations will be described
in Section \ref{b3-s3.4}, and a detailed presentation can be found in
the thesis \cite{b3-D17}.

\subsection{Multi-Waves}
\label{b3-s4.6}
For a numerical experiment of decomposing and freezing multi-waves we take up
 Example \ref{b3-exa:s4.1.1} of the Quintic Nagumo equation (QNE).
\begin{example}[Quintic Nagumo equation]\label{b3-exa:s4.1.4}
\begin{figure}[ht]
  \centering
  \subfigure[]{\includegraphics[width=0.25\textwidth]{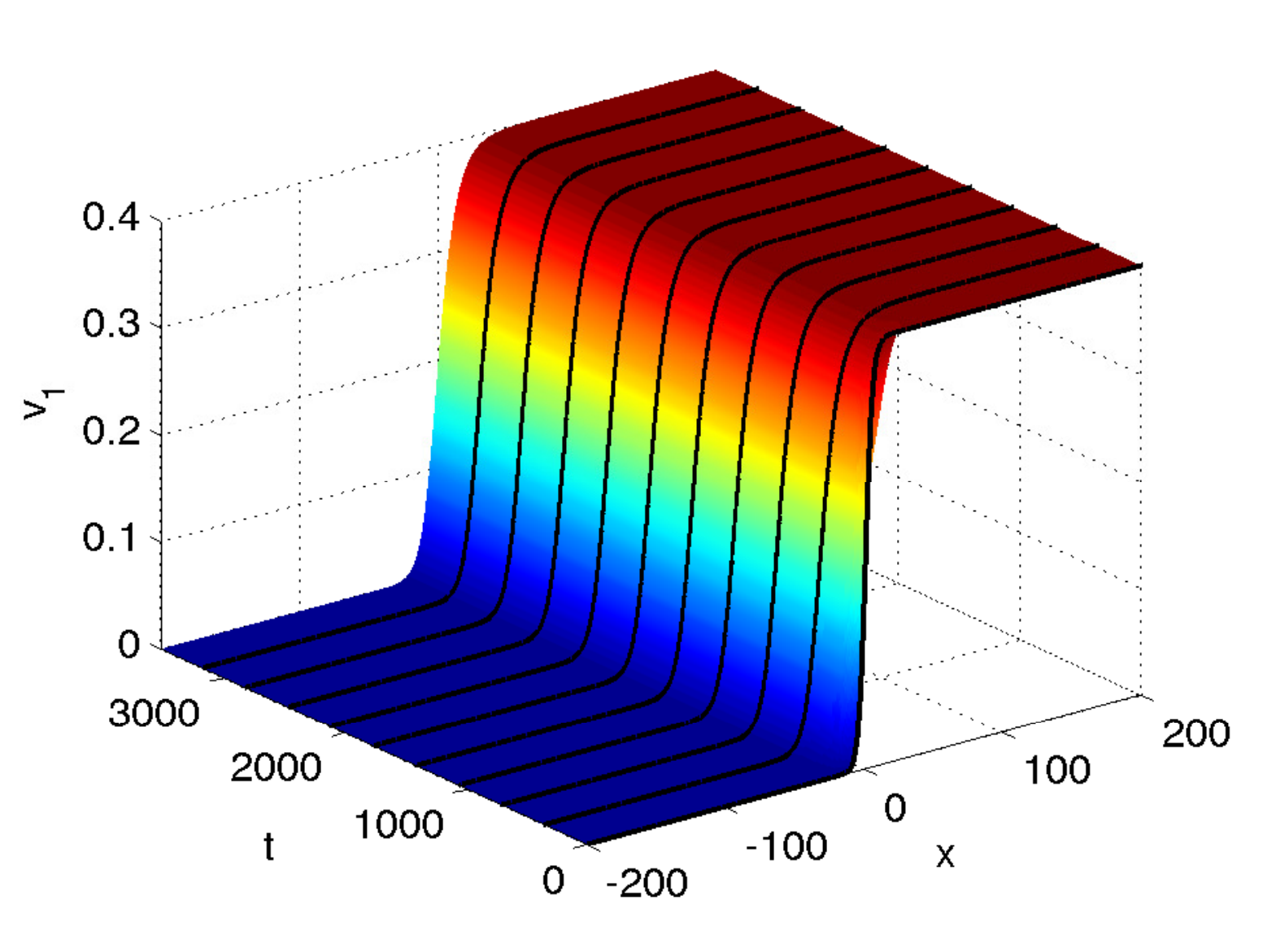}     \label{b3-fig:QN2F.a}}
  \hspace*{-0.45cm}
  \subfigure[]{\includegraphics[width=0.25\textwidth]{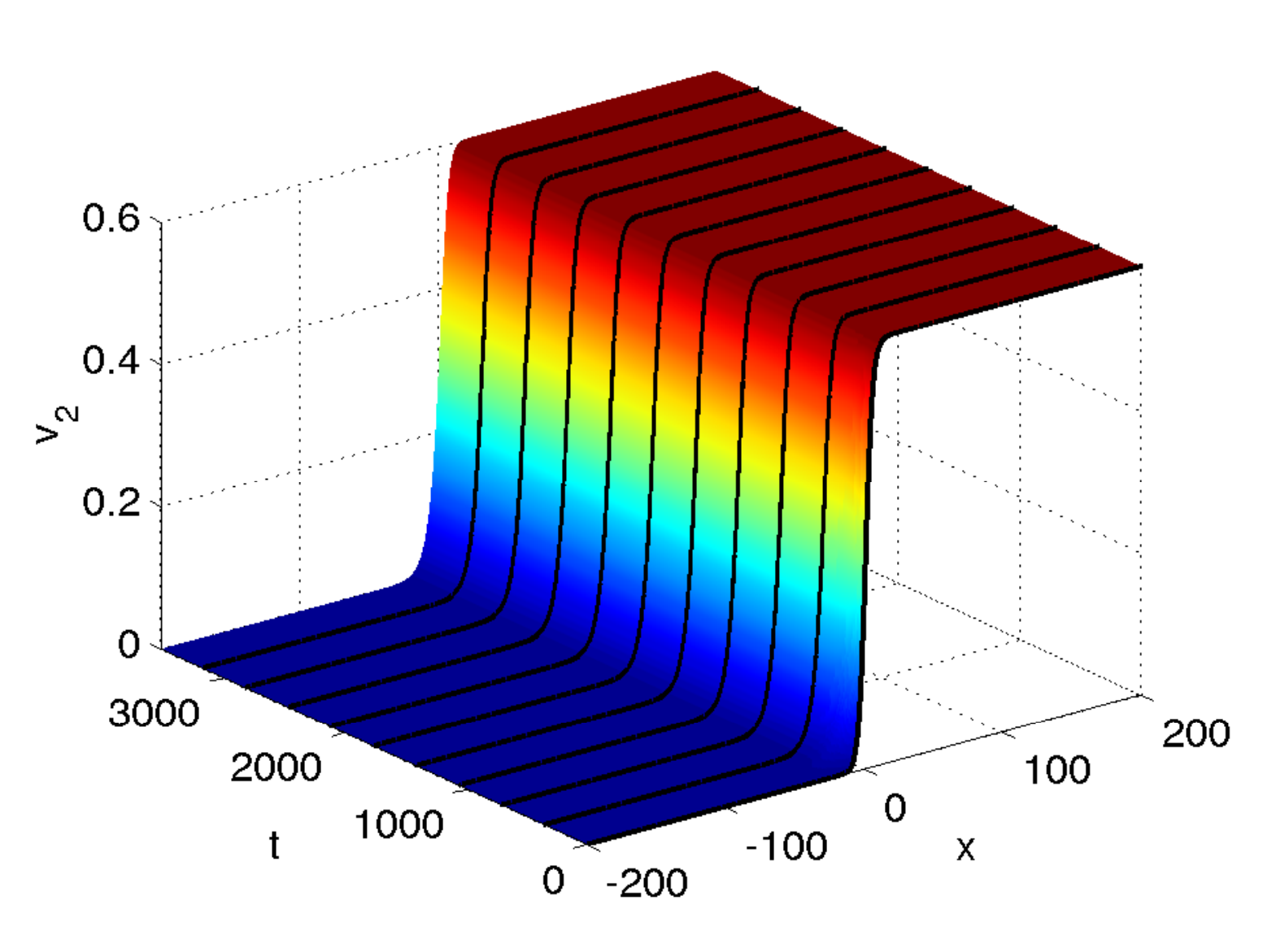}     \label{b3-fig:QN2F.b}}
  \hspace*{-0.45cm}
  \subfigure[]{\includegraphics[width=0.25\textwidth]{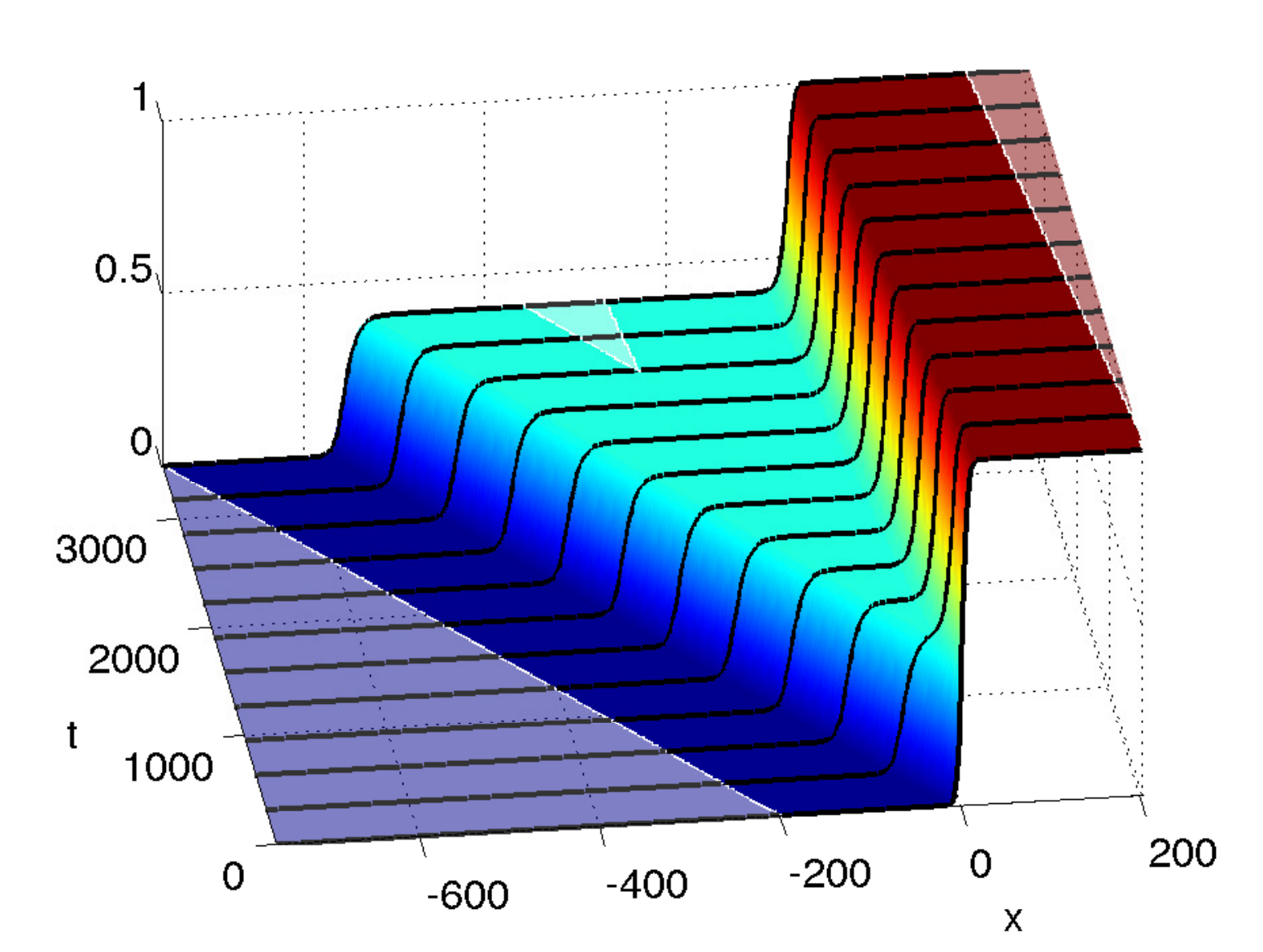} \label{b3-fig:QN2F.c}}
  \hspace*{-0.45cm}
  \subfigure[]{\includegraphics[width=0.25\textwidth]{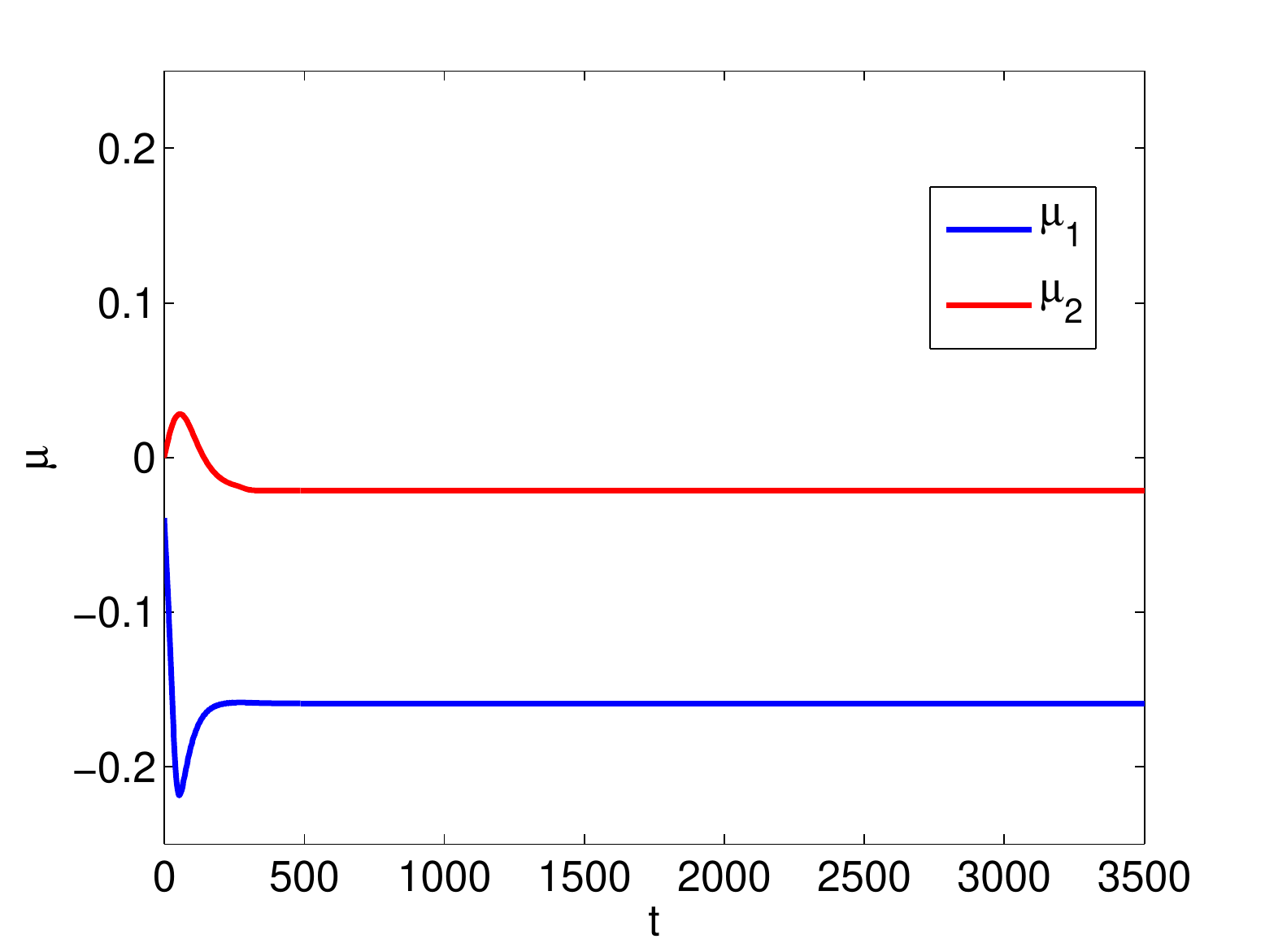}      \label{b3-fig:QN2F.d}}
  \caption{$2$-front of QNE: profile $v_1$ (a), profile $v_2$ (b), superposition (c), and velocities $\mu_1,\mu_2$ (d)}
  \label{b3-fig:QN2F}
\end{figure}

Figure \ref{b3-fig:QN2F.c} shows the time evolution of the superposition $\sum_{j=1}^{2}v_j(x-\gamma_j(t),t)$, which can be considered 
as an approximation of a travelling $2$-front $u$ of the original QNE \eqref{b3-e2.3.1} with $f$ from \eqref{b3-e3.1.1}. 
The quantities $(v_j,\mu_j)$ are the solutions of \eqref{b3-e2.3.4} and provides us approximations of $(v_{\star,j},\mu_{\star,j})$.
Figure \ref{b3-fig:QN2F.c} shows that the lower front $v_1$ (travelling at
speed $\mu_1$) is faster than the upper front $v_2$ (travelling at speed $\mu_2$), 
i.e. we may expect $\mu_{\star,1}<\mu_{\star,2}<0$. 
Figure \ref{b3-fig:QN2F.a} and \ref{b3-fig:QN2F.b} (resp. \ref{b3-fig:QN2F.d}) show the time evolution of the single front profiles $v_1$ and $v_2$ (resp. the velocities $\mu_1,\mu_2$) 
obtained by solving \eqref{b3-e2.3.4} with homogeneous Neumann boundary conditions, $f$ from \eqref{b3-e3.1.1}, parameters $b_2=\tfrac{1}{32}$, $b_3=\tfrac{2}{5}$, 
$b_4=\tfrac{73}{100}$, spatial domain $[-200,200]$, multi-waves $N=2$, initial data $v_1^0(\xi)=\tfrac{v_2^-}{2}\left(\tanh(\tfrac{\xi}{5})+1\right)$, 
$v_2^0(\xi)=\tfrac{1-v_2^-}{2}\left(\tanh(\tfrac{\xi}{5})+1\right)$ with $v_2^{-}=b_3$, $\gamma_1^0=\gamma_2^0=0$, templates $\hat{v}_j=v_j^0$, 
bump function $\varphi(\xi)=\mathrm{sech}(\tfrac{\xi}{20})$ and time range $[0,3000]$. Approximations of the single front profiles $v_{\star,j}$ 
(with $v_1^{-}=0$, $v_1^{+}=a_4=v_2^{-}$, $v_2^{+}=1$) and velocities $\mu_{\star,1}\approx -0.159$, $\mu_{\star,1}\approx -0.021$ 
are shown in Figure \ref{b3-fig:QN2F.a}, \ref{b3-fig:QN2F.b} and \ref{b3-fig:QN2F.d}. For the numerical solution of \eqref{b3-e2.3.4} we used the FEM 
for space discretisation with Lagrange $C^0$-elements and maximal element size $\triangle x=0.4$, the BDF method for time discretisation with maximum order $2$, 
intermediate time steps, time step-size $\triangle t=0.8$, and the Newton method for solving non-linear equations. 
\end{example}

\begin{figure}[ht]
  \centering
  \subfigure[]{\includegraphics[width=0.25\textwidth]{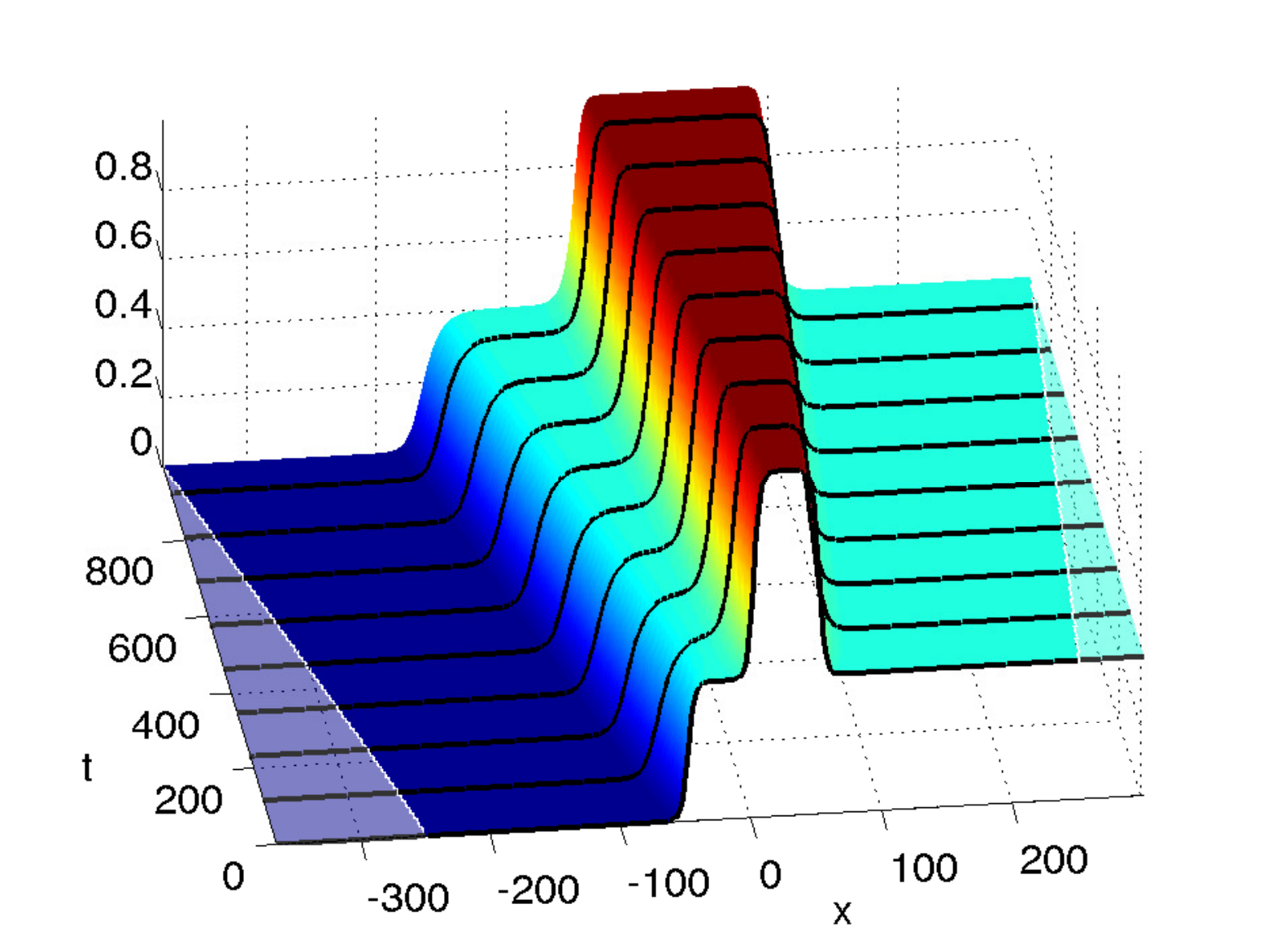}  \label{b3-fig:MW.a}}
  \hspace*{-0.45cm}
  \subfigure[]{\includegraphics[width=0.25\textwidth]{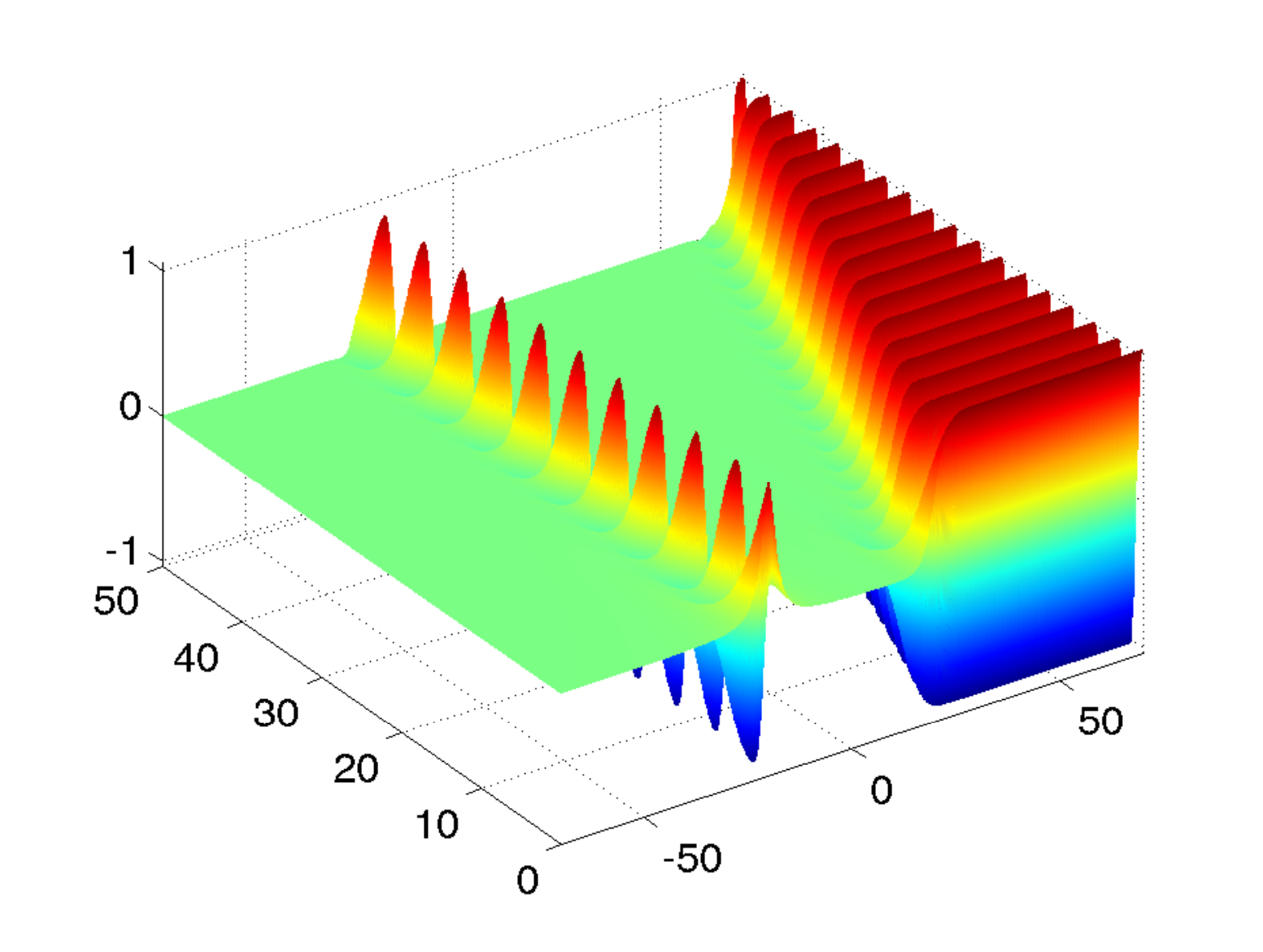}  \label{b3-fig:MW.b}}
  \hspace*{-0.45cm}
  \subfigure[]{\includegraphics[width=0.25\textwidth]{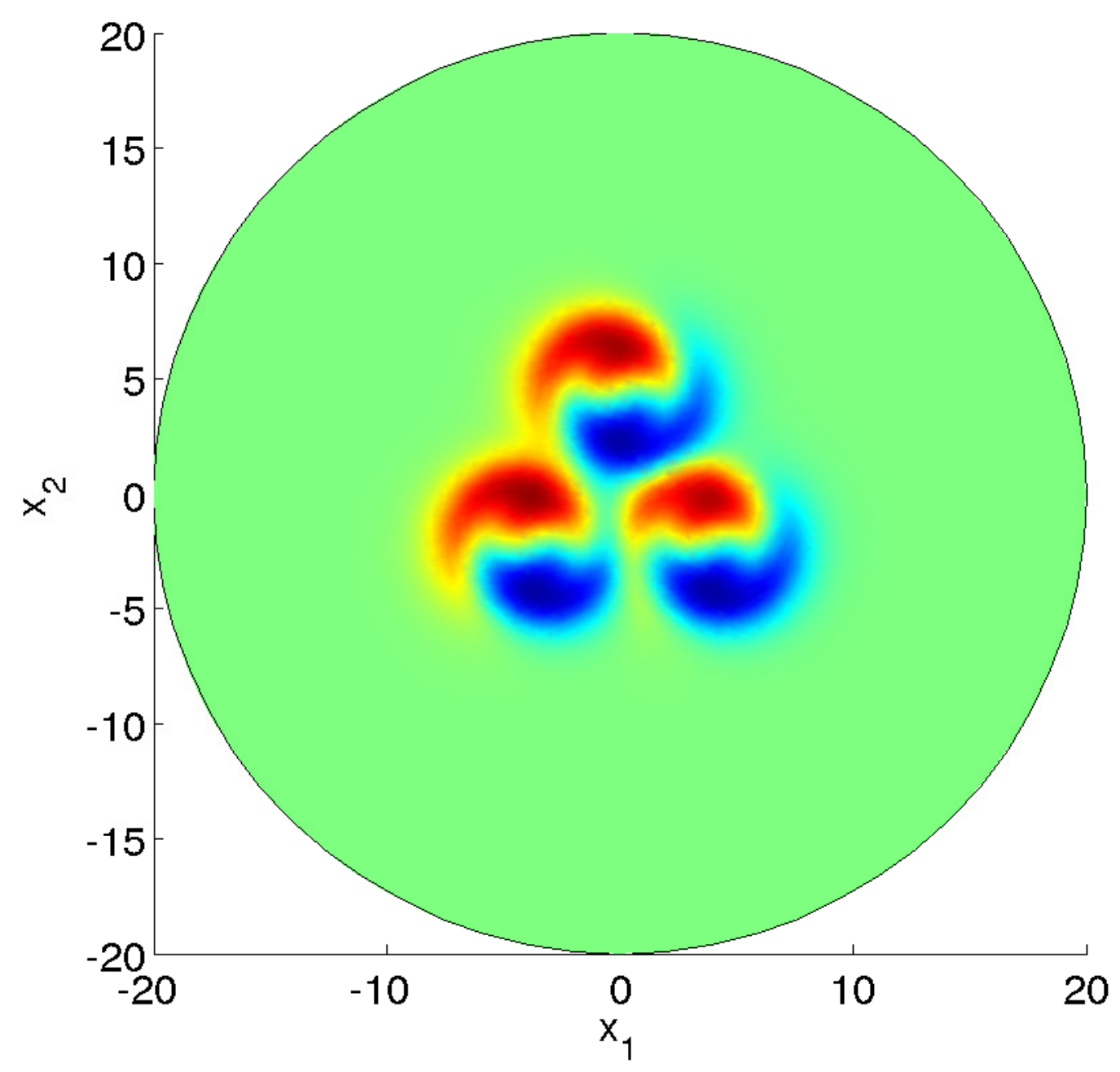}  \label{b3-fig:MW.c}}
  \hspace*{-0.45cm}
  \subfigure[]{\includegraphics[width=0.25\textwidth]{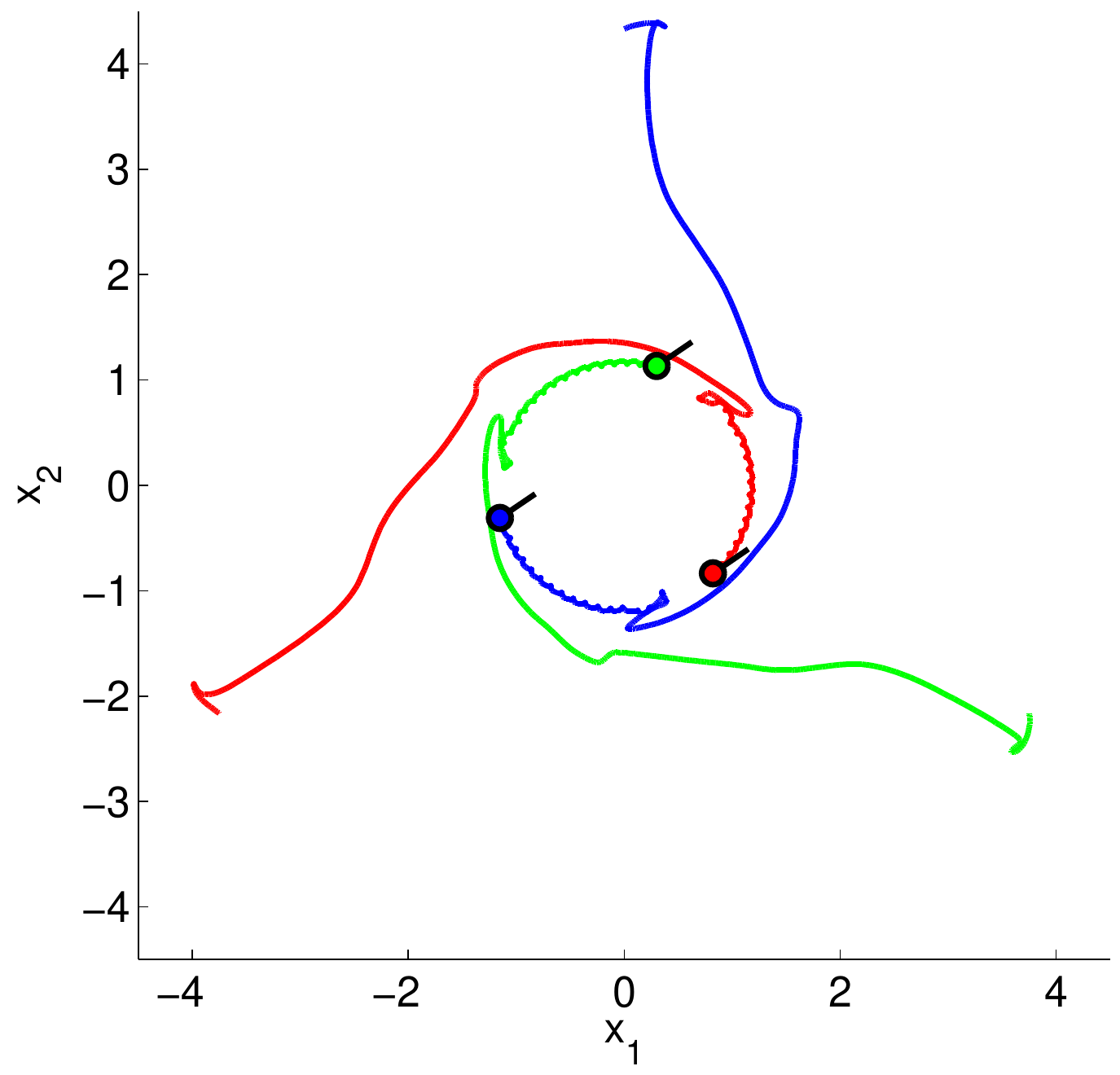}  \label{b3-fig:MW.d}}
  \caption{Multiwaves: $3$-front of QNE (a), Pulse-Front of QCGL (b), $3$-soliton of QCGL (c) and position of centres (d)}
  \label{b3-fig:MW}
\end{figure}

Travelling $2$-fronts as in Example~\ref{b3-exa:s4.1.4} are a special class of multi-waves. The decompose and freeze method (short: DFM) easily extends 
to larger numbers of fronts, e.g. $3$-fronts (see Fig.~\ref{b3-fig:MW.a}), and can be used to analyse wave interaction processes, for example repulsion and
collision of waves. 
Moreover, the DFM extends to general multi-structures, e.g. to a superposition of a phase-rotating pulse and a travelling phase-rotating front (see Fig.~\ref{b3-fig:MW.b}), 
and to higher space dimensions, see the three spinning multi-solitons in
Fig.~\ref{b3-fig:MW.c}) with the interactions represented by the traces
of their centres in Fig.~\ref{b3-fig:MW.d}. For the DFM we refer to the works \cite{b3-BOR14,b3-S09,b3-BST08}. Extensions of the DFM to rotating multi-solitons 
including numerical experiments can be found in \cite{b3-BOR14,b3-O14}.

\sect[Relative Equilibria]{Stability of Relative Equilibria}
\label{b3-s3}
The issue of stability is fundamental to all wave phenomena considered here.
Since relative equilibria come in families due to the group
action (see Section \ref{b3-s1.2}) the classical notions of (Lyapunov)-stability and asymptotic
(Lyapunov)-stability are replaced by the notions of {\it orbital stability}
and {\it stability with asymptotic phase}.
\index{stability!orbital} \index{stability!with asymptotic phase}

\subsection{Notions of stability and the co-moving frame equation}
\label{b3-s3.0}
In order to have some flexibility for the application to PDEs, the following
definition 
uses two norms $\|\cdot\|_1$ and $\|\cdot\|_2$ which need not agree with 
the norms in the Banach spaces
$Z$ and $X$. Moreover, depending on the type of PDE, a solution concept
different from the strong solution in Definition \ref{b3-d1} may
be necessary.
\begin{definition} \label{b3-d3}
A relative equilibrium $(v_{\star},\gamma_{\star})$ of \eqref{b3-e4}
is called orbitally stable with respect to norms $\|\cdot\|_1$ and
$\|\cdot\|_2$ if for any $\varepsilon > 0$ there exists $\delta >0$
such that the Cauchy problem \eqref{b3-e4} has a unique strong
solution $u$ for  $u_0\in Z$ with
$\|u_0-v_{\star}\|_1\le \delta$, and the solution satisfies
\begin{equation*} \label{b3-e4.1}
\mathrm{inf}_{\gamma \in G}\|u(t)-a(\gamma)v_{\star}\|_2 \le \varepsilon \quad
\forall t \ge 0.
\end{equation*}
It is called {\it stable with asymptotic phase} if for any $\varepsilon >0$
there exists $\delta >0$ such that for all initial values $u_0 \in Z$
with $\|u_0 - v_{\star}\|_1 \le \delta$ the Cauchy problem \eqref{b3-e4}
has a unique strong solution $u$, and
for some $\gamma_{\infty}=\gamma_{\infty}(u_0) \in G$ the solution satisfies
\begin{equation*} 
\|u(t) - a(\gamma_{\infty}\circ \gamma_{\star}(t))v_{\star}\|_2
\begin{cases} \le \varepsilon \quad \forall t \ge 0, \\
              \rightarrow 0 \quad \text{as} \quad t\rightarrow \infty.
              \end{cases}
\end{equation*}              
\end{definition}
\index{stability!orbital} \index{stability!with asymptotic phase}
Stability in general requires to investigate the solution of \eqref{b3-e4}
for initial data $u_0=u_{\star}+v_0$ which are small perturbations
of the wave profile. For this we transform into a co-moving frame via
\begin{equation*}\label{b3-e4.3}
u(t) =a(\gamma_{\star}(t))v(t), \quad t\ge0,
\end{equation*}
which by contrast to the general ansatz \eqref{b3-e20} assumes the group
orbit $\gamma_{\star}$ to be known. Instead of \eqref{b3-e21} one obtains
the {\it co-moving frame equation}
\begin{equation} \label{b3-e4.4}
v_t(t) = F(v(t)) - d_{\gamma}[a(\one)v(t)] \mu_{\star},\quad
v(0)=v_{\star}+v_0.
\end{equation}
Linearising about $v_{\star}$ in a formal sense leads to consider the linear operator
\begin{equation} \label{b3-e4.5}
\mathcal{L}w = DF(v_{\star})w - d_{\gamma}[a(\one)w]\mu_{\star}, \quad w \in Z.
\end{equation}
\index{co-moving frame equation}
If the topology on $Z$ is strong enough, then $DF$ is in fact the
Fr\'echet derivative of $F$, and this point of view is sufficient for
our applications to semi-linear PDEs in Section \ref{b3-s4}.
The general procedure then is to deduce non-linear stability in the sense of
Definition \ref{b3-d3} from spectral properties of the operator $\mathcal{L}$.
One says that the {\it principle of linearised stability} holds if such a conclusion is valid. A minimal requirement is that the spectrum lies in the left half-plane, i.e.
\begin{equation*} \label{b3-e4.6}
\sigma(\mathcal{L}) \subseteq \CC_-=\{\lambda \in \CC:\mathrm{Re}(\lambda) \le 0\}.
\end{equation*}
However the special properties of the PDEs considered here usually require more:
\begin{enumerate}
\item[(P1)]  determine eigenvalues on the imaginary axis caused by the group action,
\item[(P2)] analyse the essential spectrum
$\sigma_{\mathrm{ess}}(\mathcal{L})\subseteq \sigma(\mathcal{L})$
which arises from the loss of compactness for differential operators on unbounded domains, 
\item[(P3)]  compute isolated eigenvalues of the point spectrum
$\sigma_{\mathrm{pt}}(\mathcal{L})\subseteq \sigma(\mathcal{L})$ different from those in (P1), 
either by a theoretical or by a numerical tool.
\end{enumerate}
Let us finally note that a proof of non-linear stability becomes particularly
delicate if there is no spectral gap between the eigenvalues from (P1) and
the remaining spectrum. This occurs if the spectrum touches the
imaginary axis (wave trains, spiral waves, see
\cite{b3-DSSS09}, \cite{b3-S00}) or lies on the imaginary
axis (Hamiltonian case).

\subsection{Spectral Structures}
\label{b3-s3.1}
 Hardly
anything can be said about problems (P2), (P3) above 
within the abstract framework of equations \eqref{b3-e4},\eqref{b3-e4.4}.
However, the eigenvalues caused by symmetry have some general
structure. For this purpose recall the Lie bracket 
\index{Lie bracket}
$[\cdot,\cdot] : \fg \times \fg \rightarrow \fg$ (see e.g. 
\cite[Ch.8]{b3-FH91}) which turns $\fg=T_{\one}G$ into a Lie algebra.
The abstract definition of the bracket is in terms of the adjoint
representation $\mathrm{Ad}(g):\fg\rightarrow \fg $ of $g \in G$
given by
\begin{equation*}
\begin{aligned} 
\mathrm{Ad}(g)\nu = & d_h[g \circ h \circ g^{-1}]_{|h=\one} \nu , \quad 
\nu \in \fg, \\
[\mu,\nu]  = & d_g[\mathrm{Ad}(g)\nu]_{|g=\one}(\mu), \quad 
\mu,\nu \in \fg.
\end{aligned}
\end{equation*}
It is reasonable to look for eigenfunctions of $\mathcal{L}$ of the type
$w=d_{\gamma}[a(\one)v_{\star}]\mu, \mu \in \fg_{\CC}$, where $\fg_{\CC}$ denotes the 
complexified Lie algebra and $d_{\gamma}[a(\one)v_{\star}]$ denotes the complexified
operator. 
\begin{theorem} \label{b3-t2}
Let $v_{\star}\in Z$,$\gamma_{\star}(t)=\exp(t \mu_{\star}),t\ge 0$ be a relative equilibrium of \eqref{b3-e1} such that $d_{\gamma}[a(\one)v_{\star}]$ maps $\fg$
into $Z$.
Then $w=d_{\gamma}[a(\one)v_{\star}]\mu, \mu \in \fg_{\CC}$ solves the (complexified)
eigenvalue problem
\begin{equation} \label{b3-e4.7}
(\lambda I-\mathcal{L})w = 0
\end{equation}
if and only if $\mu$ satisfies
\begin{equation*} \label{b3-e4.8}
d_{\gamma}[a(\one)v_{\star}](\lambda \mu - [\mu, \mu_{\star}])=0.
\end{equation*}
In particular, if the stabiliser $H(v_{\star})$ is trivial (see \eqref{b3-e8a}), 
then independent eigenvectors $\mu_j,j=1,\ldots,k$ of $[\cdot,\mu_{\star}]:\fg \rightarrow \fg$
lead to independent eigenfunctions $w_j=d_{\gamma}[a(\one)v_{\star}]\mu_j,j=1,\ldots,k$
of \eqref{b3-e4.7}.
\end{theorem}
\begin{proof}
For the family of relative equilibria \eqref{b3-e8b} we have by the chain rule
\begin{align*}
F(a(\gamma(g,t))a(g)v_{\star})= & \frac{d}{dt}\left[ a(\gamma(g,t))(a(g)v_{\star})\right] \\
= & d_{\gamma}\left[a(g \circ \gamma_{\star}(t) \circ g^{-1})(a(g)v_{\star})\right]
d_h(g \circ h \circ g^{-1})_{|h=\gamma_{\star}(t)}\gamma'_{\star}(t),
\end{align*}
which upon evaluation at $t=0$ yields
\begin{align*}
F(a(g)v_{\star})= & d_{\gamma}\left[ a(\one)(a(g)v_{\star})\right]\mathrm{Ad}(g) \mu_{\star}.
\end{align*}
We differentiate with respect to $g\in G$ and apply this to $\mu\in T_{g}G$:
\begin{align*}
DF(a(g)v_{\star})d_{\gamma}[a(g)v_{\star}]\mu = &
d_{\gamma}[a(\one)(d_{\gamma}[a(g)v_{\star}] \mu)]\mathrm{Ad}(g) \mu_{\star}\\
+ & d_{\gamma}[a(\one)(a(g)v_{\star})] d_{g}[\mathrm{Ad}(g) \mu_{\star}] \mu,
\end{align*}
which upon evaluation at $g=\one, \mu \in \fg$ gives
\begin{align*}
DF(v_{\star})d_{\gamma}[a(\one)v_{\star}]\mu = &
d_{\gamma}[a(\one)(d_{\gamma}[a(\one)v_{\star}]\mu)]\mu_{\star}
+ d_{\gamma}[a(\one)v_{\star}][\mu,\mu_{\star}].
\end{align*}
Therefore, the eigenvalue problem \eqref{b3-e4.7} with
$w=d_{\gamma}[a(\one)v_{\star}]\mu$ is equivalent to
\begin{align*}
0= & \lambda w - DF(v_{\star})w + d_{\gamma}[a(\one)w]\mu_{\star} \\
= & \lambda d_{\gamma}[a(\one)v_{\star}]\mu - DF(v_{\star})d_{\gamma}[a(\one)v_{\star}]\mu
+ d_{\gamma}[a(\one)(d_{\gamma}[a(\one)v_{\star}]\mu)]\mu_{\star} \\
= & d_{\gamma}[a(\one)v_{\star}](\lambda \mu - [\mu,\mu_{\star}]),
\end{align*}
which proves our assertion.
\end{proof}
Theorem \ref{b3-t2} shows  that the geometric multiplicity
of the eigenvalue $\lambda=0$ is at least the dimension of the centraliser
of $\mu_{\star}$ given by
\begin{equation*} \label{b3-e4.9}
\fg_0(\mu_{\star}):= \{ \mu\in \fg: [\mu,\mu_{\star}] = 0 \}.
\end{equation*}
\index{centraliser}
If the group $G$ is represented as a subgroup of the matrix
group $\mathrm{GL}(\RR^N)$ for some $N \in \NN$ then the Lie bracket agrees with
the commutator. It is not difficult to see that the spectrum of the linear map
$\mu \mapsto [\mu,\mu_{\star}]$ always satisfies
\begin{equation} \label{b3-e4.10}
\sigma([\cdot,\mu_{\star}]) \subseteq \{\lambda_1 - \lambda_2: \lambda_1,\lambda_2 \in
\sigma(\mu_{\star})\} = \sigma(\mu_{\star}) - \sigma(\mu_{\star}).
\end{equation}
The special elements  $\mu_{\star}=\left( \begin{smallmatrix} S_{\star}&c_{\star}\\0 & 0 \end{smallmatrix} \right)$ from $\mathfrak{se}(d)$ (see \eqref{b3-e14})
occur with rotating waves \eqref{b3-e18} and satisfy
$\sigma(\mu_{\star}) \subseteq i \RR$ as well as $\sigma(\mu_{\star})=-\sigma(\mu_{\star})$.
Let $ \mu_1,\ldots,\mu_d $ be the eigenvalues of the skew-symmetric
matrix $S_{\star}$, then one finds
\begin{equation} \label{b3-e4.11}
\sigma([\cdot,\mu_{\star}]) = \{ \mu \in \CC: \mu \in  \sigma(S_{\star}) \;
\text{or}\; \mu=\mu_j+\mu_k\; \text{for some}\; j<k\},
\end{equation}
see \cite{b3-BI05},\cite{b3-BO16b} for the computation of eigenvalues
and corresponding eigenvectors.
\subsection{Stability with Asymptotic Phase}
\label{b3-s3.3}
We discuss sufficient conditions for the stability with asymptotic phase in case of our two model equations \eqref{b3-e9},\eqref{b3-e10} (see 
\cite{b3-BOR14},\cite{b3-H81},\cite{b3-KP13},\cite{b3-S02},\cite{b3-T08}).

For a travelling wave $(v_{\star},\mu_{\star})$ the linearised operator
$\mathcal{L}$ from \eqref{b3-e4.5} reads
\begin{equation} \label{b3-e4.12}
\begin{aligned}
\mathcal{L}w = &A w_{\xi \xi} +(\mu_{\star}I_m+ D_2f(v_{\star},v_{\star,\xi}))w_{\xi} + D_1f(v_{\star},v_{\star,\xi})w \\
= & A w_{\xi \xi} + B(\cdot)w_{\xi} + C(\cdot) w.
\end{aligned}
\end{equation}
We consider the case of a front
\begin{equation} \label{b3-e4.13}
\lim_{\xi \rightarrow \pm \infty} v_{\star}(\xi) = v_{\pm},
\quad \lim_{\xi \rightarrow \pm \infty} v_{\star,\xi}(\xi)=0,
\end{equation}
which is covered by our abstract approach only in case $v_{\pm}=0$, see Remark
\ref{b3-rem1}. Note, however, that $\mathcal{L}:H^2(\RR,\RR^m) \to L^2(\RR,\RR^m)$ is well
defined in the general case \eqref{b3-e4.13}, and that it has the eigenvalue
 $0$ with eigenfunction $w=v_{\star,\xi}$, cf. Theorem \ref{b3-t2} and
 \eqref{b3-e4.10} with $0 \in \sigma(\mu_{\star})$. The essential spectrum of
$\mathcal{L}$ is determined by the constant coefficient operators
\begin{equation} \label{b3-e4.13a}
\mathcal{L}^{\pm}= A \partial_{\xi}^2 + B_{\pm}\partial_{\xi} + C_{\pm},
\quad C_{\pm}=D_1f(v_{\pm},0),
 \quad B_{\pm}= \mu_{\star}I_m + D_2f(v_{\pm},0).
\end{equation}
Bounded solutions of $(\lambda I-\mathcal{L}^{\pm})w=0$ are of the form $w(\xi)
=e^{i \omega \xi},\omega \in \RR$ which leads to the definition of the dispersion set
\begin{equation} \label{b3-e4.14}
\sigma_{\mathrm{disp}}(\mathcal{L})= 
\left\{\lambda \in \CC: \lambda\in \sigma(-\omega^2A+i \omega B_{\pm}+C_{\pm}) \;
\text{for some sign} \; \pm \;\text{and} \; \omega \in \RR  \right\}.
\end{equation}
\index{dispersion set}
By Weyl's theorem on invariance of the essential spectrum under relatively
compact perturbations (see \cite{b3-H81},\cite{b3-KP13}) one finds
$\sigma_{\mathrm{disp}}(\mathcal{L}) \subseteq \sigma_{\mathrm{ess}}(\mathcal{L})$
and, moreover, that the connected component $U$ of
$\CC \setminus \sigma_{\mathrm{disp}}(\mathcal{L})$ containing a positive real semi-axis
satisfies $U \subseteq (\rho(\mathcal{L})\cup \sigma_{\mathrm{pt}}(\mathcal{L}))$.
Therefore, the issues (P2) and (P3) from Section \ref{b3-s3.0} are resolved
by requiring for some $\beta >0$ the  following spectral conditions
\begin{equation} \label{b3-e4.15}
\mathrm{Re}\, \sigma_{\mathrm{disp}}(\mathcal{L}) \le - \beta < 0,
\end{equation}
\begin{equation} \label{b3-e4.16}
\mathrm{Re}\left(\sigma_{\mathrm{pt}}(\mathcal{L})\setminus \{0\}\right) \le
- \beta<0 \quad
\text{and the eigenvalue} \;  0  \; \text{is simple}.
\end{equation}
A common analytical tool to verify assumption \eqref{b3-e4.16} in applications
is to study the zeroes of the so-called {\it Evans function}, see
\cite{b3-KP13},\cite{b3-S02}. For numerical purposes however, we prefer
to solve boundary eigenvalue problems subject to finite boundary conditions
and to
employ a contour method, see Section \ref{b3-s5} and \cite{b3-BLR14}.
\begin{theorem} \label{b3-t3}
Let the spectral assumptions \eqref{b3-e4.15},\eqref{b3-e4.16} above hold
and let $f$ be of the form
\begin{equation} \label{b3-e4.18}
f(u,v)= f_1(u)+ f_2(u)v, \quad f_1 \in C^2(\RR^m,\RR^m), f_2 \in
C^2(\RR^m,\RR^{m,m}).
\end{equation}
Then a travelling wave $(v_{\star},\mu_{\star})$ of \eqref{b3-e9}
is stable with asymptotic phase  for solutions in  the regularity class
$v_{\star}+ \left( C([0,\infty),H^1(\RR,\RR^m)) \cap
C^1([0,\infty),L^2(\RR,\RR^m))\right)$
and with respect to the norm $\|\cdot\|_1=\|\cdot\|_2 =
\|\cdot \|_{H^1}$.
\end{theorem}
\begin{remark} The semilinear case $f_2 \equiv 0$ is well studied, see e.g.
\cite{b3-H81},\cite{b3-KP13},\cite{b3-S02}. The more general form
\eqref{b3-e4.18} includes Burgers equation ($f_2(u)=u$) and is treated
in \cite{b3-T05}, \cite{b3-T08}. Note that the global Lipschitz conditions
imposed there can be localised via the Sobolev embedding $H^1(\RR,\RR^m)
\subset L^{\infty}(\RR,\RR^m)$.
\end{remark}
In Sections \ref{b3-s4.2} and \ref{b3-s4.3} we referred to stability results
for travelling waves in hyperbolic systems of first order
\eqref{b3-e4.2.1}, \eqref{b3-e4.2.3} and of second
order \eqref{b3-e4.3.1}. Here we consider in more detail the stability
of rotating waves for the model system \eqref{b3-e10}.
Following \cite{b3-BL08} we restrict to $d=2$ and $A=I_m$. Extensions
to $d \ge 3$ are based on \cite{b3-BO16a},\cite{b3-BO16b} and will be
indicated below. Moreover, we mention an alternative approach \cite{b3-SSW97}
towards asymptotic stability (without asymptotic phase) based on a
centre manifold reduction.

As in \eqref{b3-e18} consider a rotating wave
$v_{\star} \in H^2_{\mathrm{Eucl}}(\RR^2,\RR^m)$ centred at $x_{\star}=0$ and
with $S_{\star}= \left( \begin{smallmatrix} 0 & - \mu_{\star} \\ \mu_{\star} & 0
               \end{smallmatrix} \right), \mu_{\star} \neq 0$.
We assume decay of derivatives up to order $2$
\begin{equation*} \label{b3-e4.19}
\sup_{|\xi| \ge R} |D^{\alpha}v_{\star}(\xi)| \to 0 \quad \text{as} \; R \to \infty
\quad \text{for} \; |\alpha| \le 2
\end{equation*}
and stability of the linearisation at infinity in the sense of
\begin{equation} \label{b3-e4.20}
\mathrm{Re} \, \langle Df(0) w, w \rangle \le - \beta |w|^2 \quad \text{for all}
\; w \in \CC^m \quad \text{and some} \; \beta >0.
\end{equation}
This assumption guarantees that the essential spectrum of the linear operator
$\mathcal{L}: H^2_{\mathrm{Eucl}}(\RR^2,\RR^m) \to L^2(\RR^2,\RR^m)$ defined by
\begin{equation} \label{b3-e4.21}
\mathcal{L} v = \Delta v + \mathcal{L}_{S_{\star}}v + Df(v_{\star})v,
\end{equation}
lies in the open left half plane. As for the abstract result \eqref{b3-e4.11}
one finds that $\mathcal{L}$ has eigenvalues $0,\pm i \mu_{\star}$ with
corresponding eigenfunctions $\mathcal{L}_{S_{\star}}v_{\star}$ and
$D_1 v_{\star} \pm i D_2 v_{\star}$. The appropriate assumption
on the point spectrum of $\mathcal{L}$ then is to require that
for some $\beta >0$ (which agrees w.l.o.g with $\beta$ from  \eqref{b3-e4.20}):
\begin{equation*} \label{b3-e4.22}
 \text{The eigenvalues} \ 0,\pm i \mu_{\star} \;
\text{are simple and the only ones of} \; \mathcal{L}\; \text{with} \;
\mathrm{Re} \ge - \beta.
\end{equation*}
\begin{theorem} \label{b3-t4}
Let $f \in C^4(\RR^m,\RR^m)$ and let the rotating wave $(v_{\star},S_{\star})$
satisfy the spectral assumptions above. Then the rotating wave is
asymptotically stable with asymptotic phase  for the equation
\eqref{b3-e10} with initial data $u_0 \in H^2_{\mathrm{Eucl}}(\RR^2,\RR^m)$,
for strong solutions in the function class
$C^1([0,\infty),L^2(\RR^2,\RR^m)) \cap C([0,\infty), H^2(\RR^2,\RR^m))$, and with
respect to the norms $\|\cdot\|_1 = \|\cdot\|_{H^2_{\mathrm{Eucl}}}$,
$\|\cdot\|_2 = \|\cdot\|_{H^2}$.
\end{theorem}
Let us comment on the assumptions of this theorem and possible extensions.
In \cite[Cor.4.3]{b3-BO16a} it is shown that the derivatives $D^{\alpha}v_{\star}$,$1 \le |\alpha|\le 2$ of the solution decay even exponentially as $R \to \infty$
if \eqref{b3-e4.20} holds and if $\sup_{|\xi| \ge R}|v_{\star}(\xi)|$ falls below
a certain computable threshold. Moreover, according to \cite[Theorem 2.8]{b3-BO16b} the operator
$\lambda I - \mathcal{L}:H^2_{\mathrm{Eucl}}(\RR^2,\RR^m)\to L^2(\RR^2,\RR^m)$ is
Fredholm of index $0$ for values $\mathrm{Re}(\lambda)
> - \beta$. Hence the eigenvalues $0,\pm i \mu_{\star}$ are isolated and
of finite multiplicity. These results generalise to arbitrary space
dimensions $d \ge 3$ if the non-linearity and the solution $v_{\star}$ are
sufficiently smooth. Then it can also be shown that the eigenfunctions
which belong to eigenvalues on the imaginary axis and which are
induced by symmetry, decay exponentially
in space.
This suggests that the non-linear stability
Theorem \ref{b3-t4} genereralises to space dimensions $d \ge 3$ , but details
have not been worked out yet.
\subsection{Lyapunov Stability of the Freezing Method}
\label{b3-s3.4}
The numerical experiments in Section \ref{b3-s4} confirm for various types
of PDEs that the abstract freezing system \eqref{b3-e21},\eqref{b3-e24} has
a Lyapunov-stable equilibrium whenever the original equation \eqref{b3-e1}
has a relative equilibrium which is stable with asymptotic phase. Moreover,
one expects this property to persist under numerical approximations,
such as truncation to a bounded domain with suitable boundary conditions
as well as discretisations of space and time. In this section we  discuss
a few instances where corresponding analytical results are available.

The following result for travelling waves is taken from \cite[Theorem\,1.13]{b3-T05}.
\begin{theorem} \label{b3-t5}
Let the assumptions of Theorem \ref{b3-t3} hold and let the template function
$\hat{v}$ in \eqref{b3-equ:s4.1.3} satisfy
\begin{equation*} \label{b3-e4.2}
\hat{v} \in v_{\star}+H^2(\RR,\RR^m), \quad \langle \hat{v}_{\xi},v_{\star}-\hat{v} \rangle_{L^2}=0, \quad \langle \hat{v}_{\xi}, v_{\star,\xi} \rangle_{L^2} \neq 0.
\end{equation*}
Then the travelling wave $(v_{\star},\mu_{\star})$ is asymptotically stable for
\eqref{b3-equ:s4.1.3}. More precisely, there exist constants $\delta,C,\alpha >0$ such that \eqref{b3-equ:s4.1.3} has a unique solution $(v,\mu) $
 if $\|u_0-v_{\star} \|_{H^1} \le \delta$ and $\langle \hat{v}_{\xi},u_0 -\hat{v} \rangle_{L^2} =0$. Existence and uniqueness holds for solutions with regularity
 $\mu \in C[0,\infty)$,
$v\in C([0,\infty),H^1(\RR,\RR^m))$,
$v_t,f(v,v_{\xi}) \in C([0,\infty), L^2(\RR,\RR^m))$,
and $v(t) \in H^2(\RR,\RR^m)$ for $t >0$.
Furthermore, the following estimate is valid
 \begin{equation*}
 \|v(t) - v_{\star} \|_{H^1} + |\mu(t) - \mu_{\star}| \le
 C e^{-\alpha t} \|u_0 - v_{\star}\|_{H^1}, \quad t \ge 0.
 \end{equation*}
 \end{theorem}
 The papers \cite{b3-T08},\cite{b3-T08a} transfer these properties to a spatially discretised sytem (time is left continuous) on bounded intervals $J=[x_-,x_+]$
 with general linear boundary conditions  
 \begin{equation} \label{b3-e4.24} P_-(v(x_-)-v_-) + Q_- v_{\xi}(x_-) +
 P_+(v(x_+)-v_+) + Q_+ v_{\xi}(x_+)= 0 ,
 \end{equation}
 where $P_{\pm},Q_{\pm} \RR^{2m,m}$ and $v_{\pm}$ are given by \eqref{b3-e4.13}.
 An essential condition for stability is \cite[Hypothesis 2.5]{b3-T08a}
 \begin{equation} \label{b3-e4.25}
 \det \begin{pmatrix}
 \begin{pmatrix} P_- & Q_- \end{pmatrix}
 \begin{pmatrix} Y_-^s(\lambda) \\ Y_-^s(\lambda) \Lambda_-^s(\lambda) \end{pmatrix}
 &
 \begin{pmatrix} P_+ & Q_+ \end{pmatrix}
 \begin{pmatrix} Y_+^u(\lambda) \\ Y_+^u(\lambda) \Lambda_+^u(\lambda) \end{pmatrix}
 \end{pmatrix} \neq 0
 \end{equation}
 for all $\lambda \in \CC$ satisfying $\mathrm{Re}\lambda \ge - \beta$ and $|\lambda| \le C$
 for some large constant $C$. Here the matrices $Y_{\pm}^{s,u}(\lambda) \in \RR^{m,m}$
 are invertible and together with $\Lambda_{\pm}^{s,u}(\lambda)\in \RR^{m,m}$
 solve the quadratic eigenvalue problem (cf. \eqref{b3-e4.13a})
 \begin{equation} \label{b3-e4.26}
 AY \Lambda^2 + B_{\pm}Y \Lambda + (C_{\pm}-\lambda I_m)Y = 0
 \end{equation}
 such that $\mathrm{Re}\, \sigma(\Lambda_{\pm}^s(\lambda)) <0 <
 \mathrm{Re}\, \sigma(\Lambda_{\pm}^u(\lambda))$. Condition \eqref{b3-e4.15}  on
 the dispersion set \eqref{b3-e4.14} ensures that \eqref{b3-e4.26}
 has $m$ stable and $m$ unstable eigenvalues.
 A counterexample in \cite[Ch.5.2]{b3-T08a} shows that violation of
 \eqref{b3-e4.25} creates instabilities of the numerical solution even
 if all conditions of Theorem \ref{b3-t5} are satisfied.

We proceed with two stability results recently obtained for the freezing
formulation of the semilinear wave equation \eqref{b3-e4.3.2} and of
the NLS \eqref{b3-sd-eq01}. The assumptions on \eqref{b3-e4.3.1} are as follows
\begin{equation*} \label{b3-e4.27}
\tilde{f} \in C^3(\RR^{3m},\RR^m),
\end{equation*}
\begin{equation*} \label{b3-e4.28}
M\;  \text{is invertible}, \quad M^{-1}A \;\text{ is diagonalisable with positive eigenvalues,}
\end{equation*}
\begin{equation*} \label{b3-e4.29}
\begin{aligned}
(v_{\star},\mu_{\star}) \in & C^2_b(\RR,\RR^m)\times \RR \; \text{ is a travelling
wave of} \; \eqref{b3-e4.3.1} \; \text{with} \\
v_{\star,\xi} \in & H^3(\RR,\RR^m), \quad
\lim_{\xi \to \pm \infty}(v_{\star}, v_{\star,\xi})(\xi)= (v_{\pm},0), \;
\tilde{f}(v_{\pm},0,0) = 0,
\end{aligned}
\end{equation*}
\begin{equation*} \label{b3-e4.30}
A - \mu_{\star}^2 M \; \text{is invertible}.
\end{equation*}
The spectral assumptions concern the quadratic operator polynomial obtained
from linearising the comoving frame equation in the first line of
\eqref{b3-e4.3.2}
\begin{equation} \label{b3-e4.31}
\begin{aligned}
\mathcal{P}(\lambda,\partial_{\xi})=& M \lambda^2 - (D_3\tilde{f}(\star)+ 2 \mu_{\star} M \partial_{\xi}) \lambda - (A- \mu_{\star}^2M)\partial_{\xi}^2 - D_1\tilde{f}(\star) \\
+ & (\mu_{\star}D_3\tilde{f}(\star)-D_2\tilde{f}(\star))\partial_{\xi} , \quad
(\star) = (v_{\star},v_{\star,\xi},-\mu_{\star} v_{\star,\xi}).
\end{aligned}
\end{equation}
From this we obtain the matrix polynomials $\mathcal{P}_{\pm}(\lambda,\omega)$
by replacing the argument $(\star)$
by its limit $(v_{\pm},0,0)$ as $\xi \to \pm \infty$ and the operator
$\partial_{\xi}$
by its Fourier symbol $i \omega$. Then the dispersion set  is defined as follows
\begin{equation*} \label{b3-e4.32}
\sigma_{\mathrm{disp}}(\mathcal{P})= \{ \lambda \in \CC: \det(\mathcal{P}_{\pm}(\lambda,\omega))=0: \text{for some sign} \;\pm \; \text{and}\; \omega\in \RR
\}.
\end{equation*}
The conditions analogous to \eqref{b3-e4.15}, \eqref{b3-e4.16} are then
\begin{equation*} \label{b3-e4.33}
\mathrm{Re}\, \sigma_{\mathrm{disp}}(\mathcal{P}) \le - \beta < 0,
\end{equation*}
\begin{equation*} \label{b3-e4.34}
\mathrm{Re}\left(\sigma_{\mathrm{pt}}(\mathcal{P}(\cdot,\partial_{\xi}))\setminus \{0\}\right) \le
- \beta<0, \quad
\text{and the eigenvalue} \;  0  \; \text{is simple}.
\end{equation*}
\begin{theorem} \label{b3-t6}
Let the assumptions above be satisfied and let the template function
$\hat{v}$ in \eqref{b3-e4.3.2} fulfil
\begin{equation*} \label{b3-e4.35}
\hat{v} \in v_{\star}+H^1(\RR,\RR^m), \quad \langle\hat{v} - v_{\star},
\hat{v}_{\xi} \rangle_{L^2} = 0 , \quad \langle v_{\star,\xi},\hat{v}_{\xi}
\rangle_{L^2} \neq 0.
\end{equation*}
Then the pair $(v_{\star},\mu_{\star})$ is asymptotically stable for
the PDAE \eqref{b3-e4.3.2}. More precisely,  for all $0 < \eta < \beta$
there exist $\rho,C>0$ such that for all $u_0 \in v_{\star}+H^3(\RR,\RR^m)$,
$v_0 \in H^2(\RR,\RR^m)$, $\mu_1^0 \in \RR$ which satisfy
\begin{equation*} \label{b3-e4.36}
\|u_0 - v_{\star} \|_{H^3} + \| v_0 + \mu_{\star} v_{\star,\xi} \|_{H^2} \le \rho
\end{equation*}
as well as the consistency condition \eqref{b3-e4.3.3a}, the system
\eqref{b3-e4.3.2} has a unique solution $(v,\mu_1,\mu_2)$ with
$\mu_1 \in C^1[0,\infty)$, $\mu_2 \in C[0,\infty)$ and regularity
\begin{align*}
v- v_{\star} \in C^2([0,\infty),L^2(\RR,\RR^m))\cap
C^1([0,\infty),H^1(\RR,\RR^m))\cap C([0,\infty),H^2(\RR,\RR^m)).
\end{align*}
The following estimate holds for the solution
\begin{equation} \label{b3-e4.37}
\|v(\cdot,t) - v_{\star}\|_{H^2}+ \|v_t(\cdot,t)\|_{H^1} + |\mu_1(t) - \mu_{\star}|
\le C e^{-\eta t}( \|u_0 -v_{\star}\|_{H^3} + \|v_0 +
\mu_{\star} v_{\star,\xi}\|_{H^2}).
\end{equation}
\end{theorem}
Note that the second consistency condition \eqref{b3-e4.3.3b} does not
appear in the theorem but is used in the proof to make the acceleration
$\mu_2$ continuous at $t=0$. The proof of the theorem builds on
a careful reduction to the first order system \eqref{b3-e4.3.4} and on an
application of the stability theorem from \cite{b3-R12a}. The theory
for first order systems is also the reason for measuring the convergence
\eqref{b3-e4.37} in a weaker norm than the initial values.

Finally, we state a recent result on the Lyapunov-stability of the freezing method
for the non-linear Schr\"odinger equation \eqref{b3-sd-eq03},\eqref{b3-sd-eq04}.
It is a very special case of a general stability result from the thesis
\cite[Ch.2]{b3-D17} which applies to Hamiltonian PDEs that are equivariant
w.r.t. the action of a Lie group. The  assumptions are taken from the
abstract framework of \cite{b3-GSS90} which is a seminal paper on
the stability of solitary waves.

The following theorem is concerned with
the waves \eqref{b3-sd-eq05} for fixed values
$\mu_{\star,1},\mu_{\star,2}$ satisfying $ 4 \mu_{\star,1}> \mu_{\star,2}^2$. 
\begin{theorem} \label{b3-t7}
Let $\hat{v} \in H^3(\RR,\CC)$ be a template function such that
\begin{equation*} \label{b3-e4.38}
\langle i \hat{v}, v_{\star} \rangle_0 =0, \quad \langle \hat{v}_x, v_{\star}
\rangle_0 = 0,
\end{equation*}
\begin{equation*} \label{b3-e4.39}
\begin{pmatrix} \langle i\hat{v}, i  v_{\star} \rangle_0 &
                \langle i \hat{v}, v_{\star,x} \rangle_0 \\
                \langle \hat{v}_x, i v_{\star} \rangle_0 &
                \langle \hat{v}_x, v_{\star,x} \rangle_0
 \end{pmatrix}
 \; \text{is invertible.} 
 \end{equation*}
 Then the solitary wave $(v_{\star},\mu_{\star,1},\mu_{\star,2})$ from
 \eqref{b3-sd-eq05} is Lyapunov-stable for the system
 \eqref{b3-sd-eq03}, \eqref{b3-sd-eq04}. More precisely, for every
 $\varepsilon >0$ there exists $\delta >0$ such that the system
 \eqref{b3-sd-eq03}, \eqref{b3-sd-eq04} with $\|u_0-v_{\star}\|_{H^1} \le \delta$
 has a unique (weak) solution
 $(v,\mu_1,\mu_2)$ with $\mu_1 \in C^1[0,\infty)$, $\mu_2 \in C[0,\infty)$ and regularity
 \begin{align*}
 v \in C([0,\infty),H^1(\RR,\CC)) \cap C^1([0,\infty), H^{-1}(\RR,\CC)),
 \quad t \ge 0.
 \end{align*}
 The solution satisfies
 \begin{equation*} \label{b3-e4.40}
 \| v(\cdot,t) - v_{\star} \|_{H^1} + |\mu_1(t)-\mu_{\star,1}|
 +|\mu_2(t) - \mu_{\star,2}| \le \varepsilon, \quad t\ge 0.
 \end{equation*}
 \end{theorem}
 For the notion of weak solution employed here we refer to \cite[Ch.1.2]{b3-D17}. The proof of Theorem \ref{b3-t7} is mainly based on Lyapunov function
 techniques which are
 quite different from the semigroup and Laplace transform approaches
 used in the proofs of Theorems \ref{b3-t4}-\ref{b3-t6}.
 We also emphasise that \cite{b3-D17} contains applications to other PDEs with
 Hamiltonian structure, for example the non-linear Klein Gordon and
 the Korteweg-de Vries equation, and that spatial discretisations are also
 studied.

 \sect{Non-Linear Eigenvalue Problems}
\label{b3-s5}
In the context of this work non-linear eigenvalue problems arise when computing isolated eigenvalues of differential operators obtained by linearising about
a relative equilibrium. We refer to \eqref{b3-e4.5} for the abstract linearisation and to \eqref{b3-e4.12},\eqref{b3-e4.21},\eqref{b3-e4.31} for some examples
of operators. There are several sources of non-linearity in the eigenparameter,
see \cite{b3-GT17} for a recent survey.
Quadratic terms arise from second order equations in time \eqref{b3-e4.31},
exponential terms occur in the stability analysis of delay equations
(see \cite{b3-MN07}), and non-linear integral operators appear in the boundary element
method for linear elliptic eigenvalue problems. Of interest here is another source of non-linearity:
the use of projection boundary conditions when solving linear eigenvalue
problems for operators such as \eqref{b3-e4.12} on a bounded interval
$J=[x_-,x_+]$. In the following we summarise two of the major results
from \cite{b3-BLR14} on this problem.
\index{projection boundary conditions}

Contour methods have been developed over the last years
(\cite{b3-ASTIK09},\cite{b3-B12},\cite{b3-GT17}) and have become rather popular since no a-priori
knowledge about the location of eigenvalues is assumed.
The paper \cite{b3-BLR14} generalises the contour method
\index{contour method}
from \cite{b3-B12}  to holomorphic eigenvalue problems
\begin{equation} \label{b3-e5.1}
\mathcal{L}(\lambda)v = 0, \quad v \in X, \quad
\lambda \in \Omega \subseteq \CC,
\end{equation}
where $\mathcal{L}(\lambda): X \to Y$ are Fredholm operators of index $0$
between Banach spaces $X,Y$  which depend holomorphically on $\lambda$
in some subdomain $\Omega$ of $\CC$. The algorithm determines all eigenvalues
of \eqref{b3-e5.1} in the interior $\Omega_0 = \mathrm{int}(\Gamma)$ of some
given closed contour $\Gamma$ in $\Omega$. It is assumed that $\Gamma$ itself lies
in the resolvent set $\rho(\mathcal{L})= \{\lambda \in \CC:
N(\mathcal{L}(\lambda))= \{0\} \}$. One chooses linearly independent
elements $v_k \in Y, k=1,\ldots,\ell$ and functionals $w_j \in X^{\star}, j=1,\ldots,p$ and computes the following matrices
\begin{equation} \label{b3-e5.2}
E(\lambda) = \left( \langle w_j, \mathcal{L}(\lambda)^{-1} v_k \rangle_{j=1,\ldots,p}^{k=1,\ldots,\ell} \right) \in \CC^{p,\ell}, \quad \lambda \in \Gamma,
\end{equation}
\begin{equation} \label{b3-e5.3}
E_0 = \frac{1}{2 \pi i} \int_{\Gamma} E(\lambda) d \lambda, \quad
E_1 = \frac{1}{2 \pi i} \int_{\Gamma} \lambda E(\lambda) d\lambda.
\end{equation}
The following result from \cite[Theorem 2.4]{b3-BLR14} holds for the case of simple eigenvalues defined by the conditions
\begin{equation*} \label{b3-e5.4}
\begin{aligned}
\sigma(\mathcal{L})\cap \mathrm{int}(\Gamma) = & \{\lambda_1, \ldots,\lambda_{\varkappa} \},\\
N(\mathcal{L}(\lambda_j))=& \mathrm{span}\{x_j\}, \quad
N(\mathcal{L}(\lambda_j)^{\star})= \mathrm{span} \{ y_j\}, \quad j=1,\ldots, \varkappa, \\
\langle y_j, \mathcal{L}'(\lambda_j) x_j \rangle \neq & 0, \quad j=1, \ldots,
\varkappa.
\end{aligned}
\end{equation*}
\begin{theorem} \label{b3-t8}
Let the above assumptions hold and assume the following nondegeneracy
condition
\begin{equation} \label{b3-e5.5}
\mathrm{rank} \left( \langle w_j, x_k \rangle_{j=1,\ldots,p}^{k=1,\ldots,\varkappa}
\right) = \varkappa = \mathrm{rank} \left(
\langle y_j , v_k \rangle_{j=1,\ldots,\varkappa}^{k=1,\ldots,\ell} \right).
\end{equation}
Then $\mathrm{rank}(E_0) = \varkappa$ holds. Further let
\begin{equation} \label{b3-e5.6}
E_0 = V_0 \Sigma_0 W_0^{\star},\; V_0 \in \CC^{p,\varkappa},\;
V_0^{\star} V_0 = I_{\varkappa}, \; W_0 \in \CC^{\ell,\varkappa},\;
W_0^{\star}W_0 = I_{\varkappa}
\end{equation}
be the (shortened) singular value decomposition of $E_0$ with
$\Sigma_0 = \mathrm{diag}(\sigma_0,\ldots,\sigma_{\varkappa} )$,
$\sigma_1 \ge \sigma_2 \ge \ldots \ge \sigma_{\varkappa}>0$.
Then all eigenvalues of the matrix
\begin{equation} \label{b3-e5.7}
E_{\mathcal{L}} = V_0^{\star}E_1 W_0 \Sigma_0^{-1} \in \CC^{\varkappa,\varkappa}
\end{equation}
are simple and coincide with $\lambda_1, \ldots, \lambda_{\varkappa}$.
\end{theorem}
First note that \eqref{b3-e5.5} implies $p,\ell \ge \varkappa$, i.e.
the number of test functions and test functionals should exceed the
number of eigenvalues inside the contour. In fact, in applications
we expect to have $p  \gg \ell \gg \varkappa$.
The key of the proof is the theorem of Keldysh (see
\cite[Theorem 1.6.5]{b3-MM03}) which describes the coefficients of the meromorphic expansion of $\mathcal{L}(\lambda)^{-1}$ near its singularities in terms
of (generalised) eigenvectors. We mention that Theorem \ref{b3-t8}
generalises to eigenvalues of arbitrary geometric and algebraic multiplicity.
With the proper definition of generalised eigenvectors of \eqref{b3-e5.1}
it turns out that the Jordan normal form of the matrix $E_{\mathcal{L}}$ in \eqref{b3-e5.7}
inherits the exact multiplicity structure of the non-linear eigenvalue problem,
see \cite[Theorem 2.8]{b3-BLR14}. For the overall algorithm one approximates
the integrals in \eqref{b3-e5.3} by a quadrature rule (for analytical
contours $\Gamma$ the trapezoidal sum is sufficient since it leads to exponential convergence \cite{b3-B12}) and solves linear systems
$\mathcal{L}(\lambda) u_k = v_k,  k=1,\ldots \ell$ at the quadrature nodes
$\lambda \in \Gamma$. Note that these solutions can be used for both integrals
in \eqref{b3-e5.3}. The (shortened) singular value decomposition
\eqref{b3-e5.6} involves a rank decision revealing the number $\varkappa$
of eigenvalues inside the contour. Finally, solving the linear (!) eigenvalue
problem for the matrix $E_{\mathcal{L}} \in \CC^{\varkappa,\varkappa}$ is usually cheap
if $\varkappa$ is small.

We note that the algorithm also provides good approximations
of the eigenfunctions associated to $\lambda_j,j=1,\ldots, \varkappa$,
see \cite{b3-B12},\cite[Section 2.2]{b3-BLR14}. There is even an extension
of the contour method to cases where the nondegeneracy condition
\eqref{b3-e5.5} is violated. Then one computes some higher order moments
\begin{equation} \label{b3-e5.7a}
E_{\nu} = \frac{1}{2 \pi i} \int_{\Gamma} \lambda^{\nu} E(\lambda) d\lambda,
\quad \nu=0,1,2,\ldots
\end{equation}
and determines the eigenvalues from a suitable block Hankel matrix (see
\cite{b3-B12} for the extended algorithm and for the number of additional
integrals needed).
Numerical examples with more details on the algorithm may be found in
\cite{b3-B12}, and applications to the 
travelling waves considered here appear in \cite[Section 6]{b3-BLR14}.

Another favourable feature of the method is that the errors occurring in
the intermediate steps \eqref{b3-e5.2},\eqref{b3-e5.3},\eqref{b3-e5.6},\eqref{b3-e5.7} are well controllable. We demonstrate this for the operator
$\mathcal{L}(\lambda) = \lambda I - \mathcal{L}$ with the differential
operator $\mathcal{L}$ taken from \eqref{b3-e4.12}. The evaluation of the
matrix $E(\lambda)$
from \eqref{b3-e5.2} requires to solve inhomogeneous equations
\begin{equation} \label{b3-e5.8}
\mathcal{L}(\lambda) u = v \in L^2(\RR,\RR^m), \quad \lambda \in \Gamma \subset \rho(\mathcal{L})
\end{equation}
on a bounded interval $J=[x_-,x_+]$ with linear (but possibly $\lambda$-dependent)  boundary conditions
(cf. \eqref{b3-e4.24})
\begin{equation*} \label{b3-e5.9}
\mathcal{B}_J(\lambda)u:= P_-(\lambda)(u(x_-)-v_-)+Q_-(\lambda)u_{\xi}(x_-) + P_+(\lambda)(u(x_+)-v_+)
+ Q_+(\lambda)u_{\xi}(\lambda) =0.
\end{equation*}
Such $\lambda$-dependent boundary matrices $P_{\pm},Q_{\pm} \in C(\Omega,\RR^{2m,m})$  occur with the so-called
projection boundary conditions (\cite{b3-B90}) and lead to
fast convergence towards the solution of \eqref{b3-e5.8}
as $x_{\pm} \to \pm \infty$.
\index{projection boundary conditions}
The matrices are determined in such a way (see \cite[Section 4]{b3-BLR14})
that
\begin{equation*} 
  \begin{pmatrix}
 \begin{pmatrix} P_-(\lambda) & Q_-(\lambda)\end{pmatrix}
 \begin{pmatrix} Y_-^s(\lambda) \\ Y_-^s(\lambda) \Lambda_-^s(\lambda) \end{pmatrix}
 &
 \begin{pmatrix} P_+(\lambda) & Q_+(\lambda) \end{pmatrix}
 \begin{pmatrix} Y_+^u(\lambda) \\ Y_+^u(\lambda) \Lambda_+^u(\lambda) \end{pmatrix}
 \end{pmatrix} = I_{2m}
 \end{equation*}
holds for the matrices $Y^{s,u}_{\pm}(\lambda), \Lambda^{s,u}_{\pm}(\lambda)$ determined from
\eqref{b3-e4.26}. Condition \eqref{b3-e4.25} is then trivially satisfied.
With these preparations \cite[Cor.4.1]{b3-BLR14} reads as follows:
\begin{theorem} \label{b3-t9}
Let the assumptions of Theorem \ref{b3-t3} hold except for the condition
\eqref{b3-e4.16} on the point spectrum. Let
$\Gamma \subset \{z \in \CC: \mathrm{Re}z > -\beta\}$ ( $\beta$ from
\eqref{b3-e4.15}) be a closed contour which lies in the resolvent set of
the operator pencil
\begin{align*}
\mathcal{L}(\lambda) = \lambda I - \mathcal{L} = \lambda I
-(A \partial_{\xi}^2 + B(\cdot)\partial_{\xi} + C(\cdot))
\end{align*}
with $\mathcal{L}$ from \eqref{b3-e4.12}. Further, given linearly independent functions
$v_{k}\in L^{\infty}(\RR,\RR^m)$, $k=1,\ldots,\ell$ with compact support
 and let $w_j,j=1,\ldots,p$ be linearly independent functionals
 on $L^{\infty}(\RR,\RR^m)$ defined by
 \begin{equation*} \label{b3-e5.11}
 \langle w_j, u\rangle = \int_{\RR} \hat{w}_j(x)^{\top} u(x) dx, \quad
 \hat{w}_j \in L^1(\RR,\RR^m),\; j=1,\ldots, p.
\end{equation*}
Then for $J=[x_-,x_+]$ sufficiently large the linear boundary value problem
with projection boundary conditions
\begin{equation*} \label{b3-e5.12}
\mathcal{L}(\lambda) u_{k,J} = v_{k|J} \; \text{in}\; J, \quad \mathcal{B}_J(\lambda) u = 0
\end{equation*}
has a unique solution $u_{k,J}(\cdot,\lambda)\in H^2(J,\RR^m)$ for all $k=1,\ldots,\ell$
and $\lambda \in \Gamma$.
Moreover, for every $0 < \alpha < \beta$ there exists a constant $C>0$
such that the matrices
\begin{equation} \label{b3-e5.12a}
E_{\nu,J} = \left( \frac{1}{ 2 \pi i} \int_{\Gamma} \lambda^{\nu}
\langle \hat{w}_{j|J}, u_{k,J}(\cdot,\lambda) \rangle_{L^2(J)} d \lambda\right)_{j=1,\ldots,p}^{k=1,\ldots,\ell},  \quad \nu=0,1
\end{equation}
satisfy the estimate
\begin{equation} \label{b3-e5.13}
|E_{\nu} - E_{\nu,J}| \le C \exp( - 2 \alpha \min (|x_-|,x_+)), \quad
\nu = 0,1.
\end{equation}
\end{theorem}
Note that the integrals \eqref{b3-e5.12a} are the quantities approximating
the integrals \eqref{b3-e5.7a} over the unbounded domain.
With the estimates \eqref{b3-e5.13} at hand it is not difficult to show that the
singular values obtained in \eqref{b3-e5.6} and finally the eigenvalues
of $E_{\mathcal{L}}$ in \eqref{b3-e5.7} inherit the exponential error
estimate (see \cite[Section 4]{b3-BLR14}).

Let us finally note that the computation of isolated eigenvalues for
the linearised operator becomes
rather challenging for waves in two and more space dimensions.
We consider the contour method to be a true competitor to classical
methods for computing eigenvalues of linearisations at such profiles.

\end{document}